\documentclass[a4paper]{amsart}
\usepackage{amsmath,amsxtra,amsthm,amssymb,graphics,mathrsfs}
\usepackage{mathtools}
\usepackage[margin=2.7cm]{geometry}
\usepackage{enumerate}
\usepackage{stmaryrd}
\usepackage{indentfirst}
\usepackage{capt-of}
\usepackage{dsfont}
\usepackage{setspace}
\usepackage{caption}
\usepackage[font+=large]{subcaption}
\usepackage[all]{xy}
\usepackage{bbm}
\usepackage{graphicx}
\usepackage{xcolor}
\usepackage{verbatim} 
\usepackage[bookmarks=true, colorlinks=true,linkcolor=black, anchorcolor=black,citecolor=black, urlcolor=black, 
bookmarksopenlevel=3]{hyperref}
\usepackage{tikz}
\usetikzlibrary{patterns}
\usetikzlibrary{calc}
\usetikzlibrary{decorations.text}
\usetikzlibrary{decorations.shapes}
\usetikzlibrary{decorations.markings}
\usetikzlibrary{decorations.pathreplacing}

\numberwithin{equation}{subsection}

\newcount\tableauRow\newcount\tableauCol
\newcommand\drawtbx[3][]{%
	\begin{scope}[shift={#3}]
    \tableauRow=0
    \foreach \Row in {#2} {
       \tableauCol=1
       \foreach\k in \Row {
          \draw[#1](\the\tableauCol,\the\tableauRow)+(-.5,-.5)rectangle++(.5,.5);
          \draw(\the\tableauCol,\the\tableauRow)node{\k};
          \global\advance\tableauCol by 1
       }
       \global\advance\tableauRow by -1
    }
	\end{scope}
}

\newtheorem{theorem}{Theorem}[section]
\newtheorem{proposition}[theorem]{Proposition}
\newtheorem{lemma}[theorem]{Lemma}

\newtheorem*{proposition*}{Proposition}
\newtheorem*{theorem*}{Theorem}
\newtheorem*{conjecture*}{Conjecture}

\theoremstyle{definition}
\newtheorem{remark}{Remark}
\newtheorem{definition}[theorem]{Definition}

\newtheorem{example}[theorem]{Example}
\newtheorem{def-prop}[theorem]{Definition-Proposition}

\DeclareMathAlphabet{\mathpzc}{OT1}{pzc}{m}{it}

\DeclareMathOperator{\F}{\mathbb{F}}

\DeclareMathOperator{\Z}{\mathbb{Z}}

\newcommand{\interval}[1]{\llbracket #1 \rrbracket}

\DeclareMathOperator{\Ker}{\mathrm{Ker}}
\DeclareMathOperator{\Image}{\mathrm{Im}}
\DeclareMathOperator{\sgn}{\mathrm{sgn}}
\newcommand{\bfe}{\pmb{e}}
\newcommand{\tbx}{\mathtt{t}}

\newcommand{\aST}{\mathsf{aST}}

\newcommand{\ST}{\mathsf{ST}}
\newcommand{\pbad}{\ST^{\mathrlap{\backslash}p}}

\mathchardef\hyp="2D
\newcommand{\Fspan}{\mathbb{F}\hyp\mathrm{span}}

\makeatletter
\renewenvironment{proof}[1][\proofname]{\par
	\pushQED{\qed}%
	\normalfont \topsep0\p@\@plus6\p@ \labelsep1em\relax
	\trivlist
	\item[\hskip\labelsep\bfseries #1]\ignorespaces
	\mbox{}
}{
	\popQED\endtrivlist\@endpefalse
}
\makeatother

\makeatletter
\newcommand{\cusitem}[1]{%
\item[#1]\protected@edef\@currentlabel{#1}%
}
\makeatother

\numberwithin{equation}{subsection}

\begin{document}

\title[Irreducible representations of the symmetric groups from slash homologies]{Irreducible representations of the symmetric groups from slash homologies of $p$-complexes}
\date{\today}

\subjclass[2010]{Primary 20C30}

\keywords{Modular representation, symmetric group, permutation module, $p$-complex, slash cohomology, $p$-standard tableau}

\author{Aaron Chan}
\address{Graduate School of Mathematics, Nagoya University, Furocho, Chikusaku, Nagoya, JAPAN}
\email{aaron.kychan@gmail.com}

\author{William Wong}
\address{Graduate School of Mathematics, Nagoya University, Furocho, Chikusaku, Nagoya, JAPAN}
\email{william.wong.hy@cantab.net}

\begin{abstract}
In the 40s, Mayer introduced a construction of (simplicial) $p$-complex by using the unsigned boundary map and taking coefficients of chains modulo $p$.
We look at such a $p$-complex associated to an $(n-1)$-simplex; in which case, this is also a $p$-complex of representations of the symmetric group of rank $n$ - specifically, of permutation modules associated to two-row compositions.
In this article, we calculate the so-called slash homology - a homology theory introduced by Khovanov and Qi - of such a $p$-complex.
We show that every non-trivial slash homology group appears as an irreducible representation associated to two-row partitions, and how this calculation leads to a basis of these irreducible representations given by the so-called $p$-standard tableaux.
\end{abstract}

\maketitle

\section{Introduction}
The modern notion of $p$-complexes first appear in the works of Mayer \cite{Mayer1, Mayer2} back in the 40s.
A $p$-complex is a sequence of groups $(C_k)_{k\in \Z}$ equipped with maps, called \emph{$p$-differentials}, $(\partial_k:C_k\to C_{k-1})_{k\in \Z}$ such that any composition of $p$ consecutive such maps is zero.
A simple example is to take $C_k=\mathbb{Z}$ for all $0\leq k< p$, $\partial_k=\mathrm{id}$ for all $0<k< p$, and $C_k=0=\partial_k$ for all other $k$'s; this can be shown more visually in the form
\[
\xymatrix@C=30pt@R=0pt{
**[l]{\big( C_{p-1}} \ar[r]^{\partial_{p-1}} & C_{p-2} \ar[r]^{\partial_{p-2}} & \cdots \ar[r]^{\partial_2} & C_1 \ar[r]^{\partial_1} & C_0\big) \\
**[l]{=\big(\quad\; \Z} \ar[r]^{\mathrm{id}} & \Z \ar[r]^{\mathrm{id}}& \cdots \ar[r]^{\mathrm{id}}&\Z \ar[r]^{\mathrm{id}}& \Z\;\;\big)
}
\]

In \cite{Mayer1}, Mayer defined a notion of homology groups on a $p$-complex $C_\bullet=(C_k, \partial_k)_k$, given by ${}^aH_k:=\Ker(\partial^a)/\Image(\partial^{p-a})$ where $a\in\{1,2,\ldots, p-1\}$.
Moreover, the $p$-differential $\partial_k$ naturally induces a map  ${}^aH_k \to {}^{a-1}H_{k-1}$ for all $a$, which gives the total homology $\bigoplus_{k,a}{}^aH_k$ the structure of a $p$-complex.

It turns out that Mayer's homology theory is not quite optimal from the perspective of constructing derived category for $p$-complexes.   For example, over any prime $p>2$, the Mayer homology $\bigoplus_{k,a} {}^aH_k(C_\bullet)$ of the $p$-complex $C_\bullet:=(\Z\xrightarrow{\mathrm{id}}\Z)$ is $(C_\bullet)^{\oplus p-2}$ - this is rather redundant in the cases when $p>2$.  There are four (equivalent) ways to obtain an optimal homology of a $p$-complex, we will focus on the following two considered by Khovanov and Qi in \cite{KhQi}, which come from Khovanov's idea of categorifying quantum groups at primitive $p$-root of unity; see \cite{Mir} for the other two.

Similar to Mayer's version, at each degree $k\in \Z$, there are $p-1$ \emph{slash homology groups} and \emph{backslash homology groups}, which are subquotients of $C_k$ given respectively by
\begin{align*}
H_k^{/a} := \Ker(\partial^{a+1})/(\Image(\partial^{p-a-1})+\Ker(\partial^a))
\;\; \text{and}\;\; H_k^{\backslash a}:= (\Image(\partial^a)\cap \Ker(\partial^{p-1-a}))/\Image(\partial^{a+1}),
\end{align*}
with $a$ ranging in $\{0,1,\ldots, p-2\}$.
As in Mayer's version, the $\partial_k$'s also naturally induce a structure of $p$-complex on the (total) homology $H_\bullet := \big(\bigoplus_{k,a} H_k^{\slash a}, \overline{\partial}^{\slash}\big) = \big(\bigoplus_{k,a}H_k^{\backslash a}, \overline{\partial}^{\backslash}\big)$; see Section \ref{sec:pcpx} for details.

This article studies the slash homology of a $p$-complex associated to an $(n-1)$-simplex.
Let us be more precise about the construction of this $p$-complex.
Fix a field $\F$ of prime characteristic $p>0$ and an integer $n$.
For each $k\in\{0, 1, \ldots, n\}$, let $\Omega_k$ be the set of $k$-subsets, i.e. subsets of $\{1, 2, \ldots, n\}$ of size $k$, and $\F\Omega_k$ be the $\F$-vector space with basis $\Omega_k$.
Let $\varphi_k:\F\Omega_k\to \F\Omega_{k-1}$ be the $\F$-linear map that sends a $k$-subset $\omega$ to the formal sum of all $(k-1)$-subsets contained in $\omega$, i.e.
\[
\varphi_k:\omega \mapsto \sum_{\substack{\omega' \subset \omega \\ |\omega'|=k-1}} \omega'.
\]
It is an easy exercise to check that $\varphi_k\varphi_{k-1}\cdots \varphi_{k-(p-1)}=0$ for any $k$ (or simply $\varphi^p=0$ by omitting subscripts).
Hence, we have a $p$-complex of the form
\[
\F\Omega_\bullet:= ( \F\Omega_n \xrightarrow{\varphi_n} \F\Omega_{n-1} \xrightarrow{\varphi_{n-1}} \cdots \xrightarrow{\varphi_2}\F\Omega_1 \xrightarrow{\varphi_1}\F\Omega_0 ).
\]
When $p=2$, $\F\Omega_\bullet$ is precisely the canonical simplicial chain complex associated to an $(n-1)$-simplex, and it has trivial homology.

In fact, $\F\Omega_\bullet$ is not just a $p$-complex of $\F$-vector spaces, it is also a $p$-complex of $\mathfrak{S}_n$-representations, where $\mathfrak{S}_n$ is the symmetric group of rank $n$.
Indeed, $\F\Omega_k$ is isomorphic to the module induced from the trivial $\F(\mathfrak{S}_{n-k}\times \mathfrak{S}_k)$-module, i.e. the permutation module $M^{(n-k,k)}$ of the two-part composition $(n-k,k)$ of $n$, and $\varphi_k$'s are also $\F\mathfrak{S}_n$-linear.

Our first main result is the complete description of the slash homology of $\F\Omega_\bullet$.  
\begin{theorem*}{\rm (Theorem \ref{thm:newfinal})}
For any prime $p$ and any positive integer $n$, the total homology $p$-complex $H_\bullet := \big(\bigoplus_{k,a} H_k^{\slash a}, \overline{\partial}^{\slash}\big) = \big(\bigoplus_{k,a}H_k^{\backslash a}, \overline{\partial}^{\backslash}\big)$ is non-vanishing if, and only if, $p\neq 2$.  In such a case, it is given by
\begin{align*}
H_\bullet = & \bigoplus_{k:\,0\leq n-2k <p-1} \big( H_{n-k}^{/n-2k} \xrightarrow{\sim} H_{n-k-1}^{/n-2k-1} \xrightarrow{\sim} \cdots \xrightarrow{\sim} H_{k+1}^{/1} \xrightarrow{\sim} H_k^{/0} \big) \\
 =&  \bigoplus_{k:\,0\leq n-2k <p-1}\big( H_{n-k}^{\backslash 0} \xrightarrow{\sim} H_{n-k-1}^{\backslash 1} \xrightarrow{\sim} \cdots \xrightarrow{\sim} H_{k+1}^{\backslash n-2k-1} \xrightarrow{\sim} H_k^{\backslash n-2k} \big)\\
\cong & \bigoplus_{k:\,0\leq n-2k <p-1}\big( \underbrace{D^{(n-k,k)} \xrightarrow{\sim} D^{(n-k,k)} \xrightarrow{\sim} \cdots \xrightarrow{\sim} D^{(n-k,k)}\xrightarrow{\sim} D^{(n-k,k)}}_{n-2k+1 \text{ terms}}\big),
\end{align*}
where $D^{(n-k,k)}$ is the irreducible $\F\mathfrak{S}_n$-representation associated to the partition $(n-k,k)$.
\end{theorem*}

\smallskip 

For a partition $\lambda$ of $n$, it is well-known that the Specht module $S^\lambda$ appears a submodule of permutation module $M^\lambda$ with basis given by (the polytabloids associated to) standard $\lambda$-tableaux; see Section \ref{sec:tab} for details.  On the other hand, when $\lambda$ is so-called $p$-regular, $S^\lambda$ has a simple top $D^\lambda$ so that the multiplicity of $D^\lambda$ in $M^\lambda$ is precisely one.  Note that if $\lambda=(n-k,k)$ is a two-row partition, then it is always $p$-regular as long as $p\neq 2$.  

In view of the previous theorem, it is natural to wonder if one can write down a basis of $D^{(n-k,k)}$ in terms of tableaux combinatorics.  Note that it is already known in \cite{Kl} (and implicitly in \cite{Mat}) that when $n-2k\in\interval{0,p-2}$ and $p\neq 2$, the simple module $D^{(n-k,k)}$ has a basis indexed by the so-called $p$-standard tableaux of shape $(n-k,k)$ (see Definition \ref{def:p-std}), but no explicit description of any basis was ever given.  While investigating the homology of $\F\Omega_\bullet$, we see that such desired form of basis is indeed possible.

\begin{theorem*}{\rm (Theorem \ref{thm:basis})}
For $k$ satisfying $n-2k\in\interval{0,p-2}$, the simple $\F\mathfrak{S}_n$-module $D^{(n-k,k)}$ admits a basis
of the form\[
\big\{\,\bfe_{\tbx}+\mathrm{rad}S^{(n-k,k)}\, \big| \, \tbx\text{ is a $p$-standard tableau}  \big\},\]
where $\bfe_{\tbx}$ is the polytabloid basis element of the Specht module associated to a standard tableau $\tbx$.
\end{theorem*}

An interesting aspect in our proof is that we use a modification of the classical proof by Peel \cite{Pe}, who showed that every Specht module has a basis given by what James called \emph{standard polytabloid}, which are in bijection with standard tableaux (of the shape defining the Specht).  The key mechanism in Peel's proof is to use a certain partial order on the set of tableaux along with the so-called Garnir relations. Combining the combinatorics of $\varphi$ with Garnir relations allows us to see how a similar induction can be used to show the above result, if we consider a more refined partial order. It would be interesting to see if any similar modification would work for more general $D^\lambda$ when the notion of $p$-standard $\lambda$-tableaux makes sense.

\subsection*{Structure}
We present some details about $p$-complexes and the various homologies arising from them in Section \ref{sec:pcpx}.  Then in Section \ref{sec:newmain} we prove the first main theorem (Theorem \ref{thm:newfinal}).  In Section \ref{sec:tab} we review various results related to permutation modules and Specht modules, as well as their associated tableaux combinatorics in the case of two-row compositions.
The final Section \ref{sec:basis} is devoted to proving the second main result (Theorem \ref{thm:basis}).

\subsection*{Convention and notation}
As in \cite{Wil}, we fix throughout a field $\F$ of positive characteristic $p>0$.  By module we mean a finitely generated right module.
Maps (homomorphisms) are also applied to the right of a module.

An interval of integers will be denoted by $\interval{a,b}:=\{a,a+1,\ldots,b\}$ for $a<b\in\Z$; we also use $\interval{a}:=\interval{1,a}$ for convenience.
For a set $\mathcal{S}$, we denote by $|\mathcal{S}|$ its cardinality.
We use $\mathcal{S}\setminus \mathcal{T}$ to denote the complement of the subset $\mathcal{T}$ of $\mathcal{S}$.	

For a finite set $\Psi$, we denote by $\F\Psi$ the $\F$-vector space with basis $\Psi$.
If $\Psi$ is a finite subset of a module over some algebra, then we denote by $\Fspan\Psi$ the set of elements spanned by elements of $\Psi$ over $\F$ - we do not assume $\Psi$ is a basis of $\Fspan\Psi$ in this case, unless otherwise stated.

\section*{Acknowledgement}
Both authors were supported by JSPS International Research Fellowship.  AC is also supported by JSPS Grant-in-Aid for Research Activity Start-up program 19K23401.
This work grew out of Mark Wildon's talk at the Workshop on Representation Theory of Symmetric Groups and Related Algebras at the Institute of Mathematical Sciences, National University of Singapore, December 2017; we are thus deeply grateful for the organisers and the IMS.  

In addition, we are immensely grateful for Mark Wildon's interests, comments, and encouragement.  We also thank John Murray and Johannes Siemons for correspondence about the work \cite{BJS}.

\section{$p$-complexes and their homologies}\label{sec:pcpx}
Within this subsection, $\F$ is an arbitrary field, and $p\in \Z_{>1}$ be a natural number that is at least 2.
Suppose $A$ is an $\F$-algebra.  Consider $A$ as a graded algebra concentrated in degree $0$.  A \emph{$p$-complex of $A$-modules} is a tuple $C_\bullet = (\bigoplus_{k\in \Z} C_k, \partial)$ where $\bigoplus_{k\in\Z} C_k$ is a graded $A$-module and $\partial$ is a homogeneous endomorphism $\partial:\bigoplus_{k\in\Z} C_k\to \bigoplus_{k\in\Z} C_k$, called \emph{$p$-differential}, of degree $-1$ satisfying $\partial^p=0$.  For clarity, we use the notation $\partial_k$ to denote the restriction of $\partial$ to the degree $k$ component, i.e. \[
\partial_k:= \partial|_{C_k}: C_k\to C_{k-1}.
\]
Note that we use the \emph{chain} convention in this article instead of the usual cochain setup (like in \cite{KhQi}); hence, the exchange of subscripts and superscripts in the indices.

A natural $p$-complex version of homology groups of ordinary complexes is the following.
\begin{definition}
Let $C_\bullet$ be a $p$-complex of $A$-modules.
The $r$-th \emph{$p$-homology} at degree $k\in \Z$ of $C_\bullet$, where $r\in\interval{p-1}$, is defined as the  $A$-module
\[
{}^rH_k(C_\bullet) := \Ker(\partial_k^r) / \Image(\partial_{k+p-r}^{p-r}).
\]
\end{definition}

We are, however, more interested in the following homology groups.

\begin{definition}
Let $C_\bullet$ be a $p$-complex of $A$-modules.
The \emph{$a$-th slash homology group and backslash homology group at degree $k\in\Z$ of $C_\bullet$}, where $a\in\interval{0,p-2}$, are defined as the $A$-modules
\begin{align*}
H_k^{/a}(C_\bullet) &:= \Ker(\partial^{a+1}_k)/(\Image(\partial^{p-a-1}_{k+p-a-1})+\Ker(\partial^a_k)),\\
\text{and }\quad H_k^{\backslash a}(C_\bullet) &:= (\Image(\partial^a_{k+a})\cap \Ker(\partial^{p-a-1}_k))/\Image(\partial^{a+1}_{k+a+1}),
\end{align*}
respectively.
\end{definition}

It is clear that $H_k^{/0}(C_\bullet)={}^1H_k(C_\bullet)$ and $H_k^{\backslash 0}(C_\bullet)={}^{p-1}H_k(C_\bullet)$ for any $p$-complex $C_\bullet$.  In fact, as remarked in \cite{KhQi}, the natural inclusion $\iota_{k}^{(r-1)}:\Ker(\partial_k^{r-1})\to \Ker(\partial_k^r)$ for any $k\in\Z$ and $r\in \interval{p-1}$ induces an exact sequence
\begin{equation}\label{eq:4termseq}
0 \to H_k^{\backslash p-r}(C_\bullet) \to {}^{r-1}H_k(C_\bullet) \xrightarrow{\overline{\iota}_{k}^{(r-1)}} {}^rH_k(C_\bullet) \to H_k^{/r-1}(C_\bullet) \to 0,
\end{equation}
where ${}^\ell H_k(C_\bullet)$ is treated as for $\ell\in\{0,p\}$.

Observe that the restriction of the $p$-differential $\partial_k:\Ker(\partial_k^{r})\to \Ker(\partial_{k-1}^{r-1})$ induces $\overline{\partial}: {}^rH_k(C_\bullet)\to {}^{r-1}H_{k-1}(C_\bullet)$.
Then it is clear from the definition that the $\overline{\partial}$'s commute with $\overline{\iota}$'s, and give rise to the following commutative diagram 
\begin{equation}\label{eq:4termseq2}
\vcenter{\xymatrix{ 0 \ar[r]& 
H_k^{\backslash p-r} \ar[r]\ar[d]_{\overline{\partial}^{\backslash}}& 
{}^{r-1}H_k \ar[rr]^{\overline{\iota}_{k}^{(r-1)}} \ar[d]_{\overline{\partial}} &&
{}^rH_k \ar[r] \ar[d]_{\overline{\partial}} &
H_k^{/r-1} \ar[r]\ar[d]^{\overline{\partial}^{\slash}} & 0 \\ 
0 \ar[r]& H_{k-1}^{\backslash p-r+1} \ar[r]& {{}^{r-2}H_{k-1}}  \ar[rr]^{\overline{\iota}_{k-1}^{(r-2)}} &&{}^{r-2}H_{k-1} \ar[r]& H_{k-1}^{/r-2} \ar[r]& 0,
}}
\end{equation}
for all $r\in\interval{2, p-1}$, with both rows being exact.

\begin{definition}
The \emph{slash homology} and \emph{backslash homology} at degree $k$ are the $p$-complexes
\[
H_k^{\slash}(C_\bullet):=\left( \bigoplus_{a=0}^{p-2} H_{k+a}^{\slash a}(C_\bullet)\,,\; \overline{\partial}^{\slash}\right) \quad\text{and}\quad H_k^{\backslash}(C_\bullet):=\left( \bigoplus_{a=0}^{p-2} H_{k-a}^{\backslash a}(C_\bullet)\,,\; \overline{\partial}^{\backslash}\right) 
\]
respectively.  The \emph{total homology} $H_\bullet(C_\bullet)$ is the direct sum of all slash (or equivalently, backslash) homologies over all degrees.
\end{definition}
Note that $(\overline{\partial}^{\slash})^{p-1}=0=(\overline{\partial}^{\backslash})^{p-1}$ and $H_k^{\slash}(C_\bullet)$ is concentrated in degree $k, k+1, \ldots, k+p-2$, whereas $H_k^{\backslash}(C_\bullet)$ is concentrated in degree $k-p+2, k-p+3, \ldots, k-1, k$.  We remark also that $H_k^{\slash}(C_\bullet)$ is in general different from $H_{k+p-2}^{\backslash}(C_\bullet)$.

If we view the category of $p$-complexes of vectors spaces as the category of graded modules over $\F[\partial]/(\partial^p)$, then the $a$-th slash homology group $H_k^{/a}(C_\bullet)$ picks up (or mnemonically `slashes through'\footnote{This is not the original reason why they are called slash in \cite{KhQi}, but we find this mnemonic useful.}) the $a$-th radical layer of any indecomposable non-projective direct summand of $C_\bullet$ occurring at degree $k$, whereas the $a$-th backslash homology group picks up the $a$-th socle layer (i.e. slashes through the $a$-th layer from the back). 
In particular, the total slash homology of a complex corresponds to removing the maximal projective direct summand.
In contrast, $\bigoplus_{a,k} {}^aH_k(C_\bullet)$ is almost always much larger than the total slash homology.  For example, as we have mentioned in the introduction, the total homology arising from $p$-homology group for $C_\bullet = (\F \xrightarrow{ \mathrm{id}} \F)$ is $C_\bullet^{\oplus p-2}$, whereas the total homology arising from slash or backslash homology group will be $C_\bullet$ itself whenever $p>2$.

\section{Proof of first main theorem}\label{sec:newmain}

From now on we will consider the $p$-complex $\F\Omega_\bullet := (\bigoplus_{k\in \Z} \F\Omega_k, \varphi)$ as defined in the introduction.  In particular, the degree $k$ component $\F\Omega_k$ is defined as $0$ for all $k\notin \interval{0,n}$ and $\varphi_k=0$ at degree $k\notin\interval{n}$.

The aim of this section is to prove the first main theorem.  We start by recalling some results from \cite{BJS}.

For $k\in\interval{0,n}$ and $r\in\interval{p-1}$, we define an integer
\[
 h(k,r) := 2k-n+p-r.
\]
Note that in \cite{BJS}, $r$ and $h(k,r)$ are denoted by $i$ and $a$ respectively.  It will be helpful to keep in mind that our convention for $a,r$ are the indices for slash homology and $p$-homology respectively; in particular, we can switch between these indices via $a=r-1$.

\begin{theorem}\label{thm:BJS}
The following holds for any $k\in\interval{0,n}$ and $r\in\interval{p-1}$.
\begin{enumerate}
\item {\rm \cite[Theorem 3.2]{BJS}} If ${}^{r}H_k\neq 0$, then $h(k,r)\in\interval{p-1}$.
\item {\rm \cite[Theorem 6.5]{BJS}} Suppose we have $p>2$.  If $r\leq k$ and $h(k,r)=p-1$, then the $r$-th $p$-homology ${}^{r}H_k$ is isomorphic to the irreducible representation $D^{(k,k-r+1)}$ of $\F\mathfrak{S}_n$.
\end{enumerate}
\end{theorem}

Let us refine Theorem \ref{thm:BJS} (2) slightly.
\begin{lemma}\label{BJS-supplement}
Consider two integers $r\in\interval{p-1}$ and $k\in\interval{0,n}$.
If $r>k$ and $h(k,r)=p-1$, then $(k,r)=(n,n+1)$ and $p>n$.
In particular, the condition $r\leq k$ in Theorem \ref{thm:BJS} (2) is not necessary.
\end{lemma}
\begin{proof}
By the definition of $h(k,r)$ and the assumption of $h(k,r)=p-1$, we have $r=2k-n+1$.  Clearly, we must have $h(n,n+1)=p-1$; hence, the assumption on $r$ implies that $p>n$.  On the other hand, if $k<n$, say $k+c=n$ for some integer $c\geq 1$, then $r=2k-k-c+1 = k-c +1 \leq k$.

For $(k,r)=(n,n+1)$, we have $\varphi_k^r = \varphi_n^{n+1} = 0$, and so ${}^{r}H_n\cong \F\Omega_n$.  Since $p>n$, $\varphi_n^n:\F\Omega_n\to \F\Omega_0$ is a non-zero $\F\mathfrak{S}_n$-module homomorphism whose domain and codomain are both one-dimensional.  Hence, $\varphi_n^n$ is an isomorphism.  As $\F\Omega_0$ is just the trivial $\F\mathfrak{S}_n$-module $D^{(n)}$, and $(k,k-r+1)= (n,n-n-1+1)=(n,0)$, the second part of the claim follows.
\end{proof}

\begin{example}
We refer to Figure \ref{fig:n,p=12,7} for the case $(n,p)=(12,7)$.
Note that the basis used for the plot are $k$ and $a$; the indices for slash homology rather than those of $p$-homology (which is $k$ and $r$ in our convention).  We note also that the $a$-axis is tilted so that the induced $p$-differential $\overline{\varphi}^\slash$ appears horizontal.  All $k,r$ satisfying $h(k,r)=h(k,a+1)\in \interval{p-1}$ are all the points plotted in the figure.  The points labelled by an irreducible representation are those with $h(k,a+1)=p-1$.
\end{example}

\begin{figure}
\centering
\newcommand{\cpx}[4]{ 
\foreach \k in {0,...,#2}{ \node (#3v\k) at ($(#1)+(\k,0)$) {#4}; }
\pgfmathsetmacro{\arrowcount}{int(#2-1)}
\foreach \k [count=\kk] in {0,...,\arrowcount}{ \draw[->] (#3v\k) edge (#3v\kk); }
}
\begin{tikzpicture}
\node[draw] at (-11,7) {$n=12, p=7$};

\node at (-12.8,0) {$k$};
\draw[-latex ,gray] (0.2,0) -- (-12.5,0);
\node[draw] at (0,-0.4) {\footnotesize 0};
\foreach \k in {1, ...,12} {
	\draw ({-\k}, 0.1) -- +(0, -0.2);
	\node[draw] at ({-\k}, -0.4) {\footnotesize \k};
}
\node at (-3.2,-3) {$a$};
\draw[-latex, gray] (0.2,0.2) -- (-3,-3);
\foreach \a in {0, ...,5} {
	\node[draw, circle, inner sep=1pt] at ({-\a/2}, {-\a/2}) {\footnotesize \a};
	\draw[dotted] ({-\a/2},{-\a/2}) -- +(-6.3,6.3);
}
\node at (-6.2,7.2) {\small $h(k,a+1)$};
\draw[-latex, gray] (-2,3) -- (-6,7);
\foreach \h in {1, ..., 6} {
	\node[draw,circle, inner sep=1pt] at ($(-2.4,3.6) + (-0.5*\h, 0.5*\h)$) {\footnotesize \h};
	\draw[dotted] ($(-2.5,3.5) + (-0.5*\h, 0.5*\h)$) -- +(-3.3,-3.3);
}

\node (r1) at (-6,6) {$\bullet$};
\cpx{-7,5}{2}{r2}{$\bullet$};
\cpx{-8,4}{4}{r3}{$\bullet$};
\draw[dashed, ->] (r1) edge (r2v1) (r2v0) edge (r3v1) (r2v1) edge (r3v2) (r2v2) edge (r3v3);

\cpx{-8,3}{4}{r4}{$\times$};
\draw[dashed, ->] (r3v0) edge (r4v0)  (r3v1) edge (r4v1) (r3v2) edge (r4v2) (r3v3) edge (r4v3) (r3v4) edge (r4v4);
\cpx{-7,2}{2}{r5}{$\times$};
\draw[dashed, ->] (r4v1) edge (r5v0) (r4v2) edge (r5v1) (r4v3) edge (r5v2);
\node (r6) at (-6,1) {$\times$};
\draw[dashed, ->] (r5v1) edge (r6);

\draw[->] (-3,6.5) -- +(1,0) node [midway,above] {\small $\overline{\varphi}$ or $\overline{\varphi}^{\slash}$};
\draw[dashed,->] (-2.5,6.2) -- +(0,-1) node[midway,right] {\small $\overline{\iota}_k^{(a+1)}$};

\draw [decorate,decoration={brace,amplitude=10pt},xshift=-4pt,yshift=0pt]
(-8.5,3.6) -- (-8.5,6.4) node [midway, align=right, xshift=-2.2cm] 
{\footnotesize $a+1\leq h(k,a+1)\leq p-1$};

\draw [decorate,decoration={brace,amplitude=10pt},xshift=-4pt,yshift=0pt]
(-8.5,0.5) -- (-8.5,3.4) node [midway, align=right, text width=1.8cm, xshift=-1.5cm] 
{\footnotesize ${}^{a-1}H_k\neq 0$ $H_k^{\slash a}=0$};

\node[fill=white, inner sep=0pt] at (-6,6) {\scriptsize $D^{(6,6)}$};
\node[fill=white, inner sep=0pt] at (-7,5) {\scriptsize $D^{(7,5)}$};
\node[fill=white, inner sep=0pt] at (-8,4) {\scriptsize $D^{(8,4)}$};
\end{tikzpicture}
\caption{Position of non-vanishing homology groups for $(n,p)=(12,7)$.}\label{fig:n,p=12,7}
\end{figure}
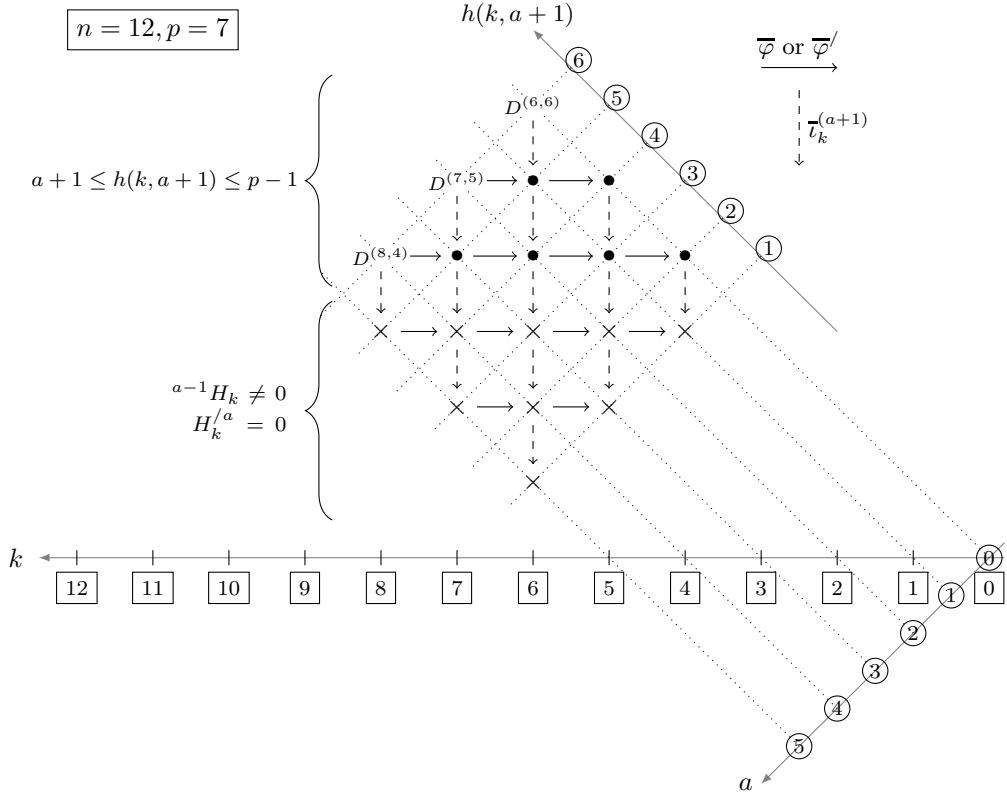

The main feature shown in the proof of \cite[Theorem 6.9]{BJS}, i.e. the calculation of composition factors of ${}^rH_k$, is that one can build these composition factors inductively from those with $h(k,r)=1$ until (and excluding) $h(k,r)>r$ (and a dual inductive mechanism from $h(k,r)=p-1$ to $h(k,r)<r$).  We will explain in the following the two key lemmas that are used to build this induction.  Note that statement (3) of both of the following lemmas are not explicitly stated in \cite{BJS} but is an argument used implicitly in \cite[Proof of Theorem 6.9]{BJS}; they are immediately from combining statement (1) and (2) in the respective case.

Recall there are natural maps
\[
\overline{\iota}_{k}^{(r)} : {}^rH_k \to {}^{r+1}H_k \quad \text{ and } \overline{\varphi} : {}^rH_k \to {}^{r-1}H_{k-1}
\]
induced by the natural inclusion $\iota_k^{(r)}:\Ker(\varphi_k^r)\hookrightarrow \Ker(\varphi_k^{k+1})$ and the restriction of the $p$-differential to $\Ker(\varphi_k^r)$
respectively.

\begin{lemma}\label{lem:BJS-iota}
For any $p>2$, $k\in\interval{0,n}$ and $r\in\interval{p-1}$ such that $h(k,r)\in\interval{p-1}$, the following hold.
\begin{enumerate}
\item {\rm\cite[Lemma 6.6]{BJS}} If $r+1\geq h(k,r)$, then $\overline{\iota}_{k}^{(r)}: {}^rH_k \to {}^{r+1}H_k$ is surjective.

\item {\rm\cite[Lemma 6.7]{BJS}} If $r\leq h(k,r)$, then $\overline{\iota}_k^{(h(k,r)-r)}: {}^rH_k \to {}^{h(k,r)}H_k$ is an isomorphism.

\item If $r\leq h(k,r)$, then $\overline{\iota}_k^{(r)}: {}^rH_k \to {}^{r+1}H_k$ is injective.  In particular, ${}^rH_k$ is a direct summand of ${}^{r+1}H_k$.
\end{enumerate}
\end{lemma}

It will be helpful to understand this lemma using Figure \ref{fig:n,p=12,7}, namely, it says that the vertical arrows connecting the black points (i.e. the upper four) are injective, those connecting the crossed points (i.e. the lower four) are surjective, and those from black to crossed points (i.e. the middle five) are bijective.  Note that when $n$ is odd, then we can only say that any vertical arrow from a `black point' to a `crossed point' is surjective.

\begin{lemma}\label{lem:BJS-del}
For any $p>2$, $k\in\interval{0,n}$, and $r\in\interval{p-1}$ such that $0<h(k,r)<p$, the following hold.
\begin{enumerate}
\item {\rm\cite[Lemma 6.6]{BJS}} If $r\leq p-h(k,r)+1$, then $\overline{\varphi} : {}^rH_k \to {}^{r-1}H_k$ is surjective.

\item {\rm\cite[Lemma 6.8]{BJS}} If $2k\leq n$, then $\overline{\varphi}^{n-2k}: {}^{r}H_{n-k} \to {}^{r-(n-2k)}H_k$ is an isomorphism.

\item If $2k>n$ and $r>1$, then $\overline{\varphi}: {}^{r}H_k\to {}^{r-1}H_{k-1}$ is injective.  In particular, ${}^{r}H_k$ is a direct summand of ${}^{r-1}H_{k-1}$.
\end{enumerate}
\end{lemma}

Let us interpret this via Figure \ref{fig:n,p=12,7} again.  The horizontal arrows on the left-hand side in this diamond formed by the plotted points are the ones where Lemma \ref{lem:BJS-del} (3) applies; whereas those on the right-hand side are the ones where Lemma \ref{lem:BJS-del} (1) applies.  Lemma \ref{lem:BJS-del} (2) means that the endpoints of each horizontal line are isomorphic modules.

\begin{lemma}\label{lem:del-on-slash}
For any $p>2$, $k\in\interval{0,n}$, and $r\in\interval{p-1}$ such that $1\leq r\leq h(k,r)< p-1$, the induced $p$-differential $\overline{\varphi}^{\slash}: H_{k+1}^{\slash r}\xrightarrow{\sim} H_k^{\slash r-1}$ is an isomorphism.
\end{lemma}
\begin{proof}
Applying Lemma \ref{lem:BJS-iota} (3) and Lemma \ref{lem:BJS-del} (3) to the diagram \eqref{eq:4termseq2} yields the commutative diagram
\[\xymatrix@C=40pt{
& 0\ar[d] & 0\ar[d] & & \\
0 \ar[r] & {}^{r}H_{k+1} \ar[d]_{\overline{\varphi}} \ar[r]^{\overline{\iota}_{k+1}^{(r)}}
 & {}^{r+1}H_{k+1} \ar[d]_{\overline{\varphi}}\ar[r]  & H_{k+1}^{\slash r}\ar[d]_{\overline{\varphi}^{\slash}}^{\wr}\ar[r]&0\\
0 \ar[r] & {}^{r-1}H_k \ar[d] \ar[r]^{\overline{\iota}_{k}^{(r-1)}}   & {}^rH_k \ar[r]\ar[d] & H_k^{\slash r-1} \ar[r] & 0 \\
&  \mathrm{Coker}(\overline{\varphi}|_{{}^rH_{k+1}}) \ar[r]^{\sim}\ar[d] & \mathrm{Coker}(\overline{\varphi}|_{{}^{r+1}H_{k+1}}) \ar[d] & & \\
& 0  & 0 ,
}\]
where the rows and columns are all exact.  The claim now follows.
\end{proof}

\begin{theorem}\label{thm:newfinal}
For any prime $p$ and any positive integer $n$, the slash homology $p$-complex $H_k^/$ (resp. $H_{n-k}^{\backslash}$) is non-vanishing if, and only if, $p\neq 2$ and $n-2k\in\interval{0, p-2}$.  In such a case, we have
\begin{align*}
H_k^{/} &\cong  \big( H_{n-k}^{/n-2k} \xrightarrow{\sim} H_{n-k-1}^{/n-2k-1} \xrightarrow{\sim} \cdots \xrightarrow{\sim} H_{k+1}^{/1} \xrightarrow{\sim} H_k^{/0} \big) \\
&\cong \big( D^{(n-k,k)} \xrightarrow{\sim} D^{(n-k,k)} \xrightarrow{\sim} \cdots \xrightarrow{\sim} D^{(n-k,k)}\xrightarrow{\sim} D^{(n-k,k)}\big) \\
& \cong  \big( H_{n-k}^{\backslash 0} \xrightarrow{\sim} H_{n-k-1}^{\backslash 1} \xrightarrow{\sim} \cdots \xrightarrow{\sim} H_{k+1}^{\backslash n-2k-1} \xrightarrow{\sim} H_k^{\backslash n-2k} \big) \;\cong\; H_{n-k}^{\backslash}.
\end{align*}
In particular, the following hold.
\begin{enumerate}
\item $H_k^{/a}$ is non-vanishing if, and only if, $p\neq2$ and  $n-2(k-a)\in\interval{a,p-2}$; in such a case, we have $H_k^{\slash a}\cong D^{(n-(k-a), k-a)}$.

\item $H_{k}^{\backslash a}$ is non-vanishing if, and only if, $p\neq2$ and  $2(k+a)-n\in\interval{a,p-2}$; in such a case, we have $H_k^{\backslash a} \cong D^{(k+a, n-(k+a))}$
\end{enumerate}
\end{theorem}
\begin{proof}
For $p=2$, we know already that $\F\Omega_\bullet$ is exact and so all homologies vanish. We only need to consider the case when $p\neq 2$.

For convenience of translating results shown before, we take $r:=a+1\in\interval{p-1}$.  
Recall that $H_k^{\slash r-1} = \mathrm{Coker}(\overline{\iota}_{k}^{(r-1)})$ from the four-term exact sequence \eqref{eq:4termseq}.  So $H_k^{\slash r-1}\neq 0$ implies that ${}^rH_k\neq 0$ as well as $\overline{\iota}_{k}^{(r-1)}$ is not surjective.  Thus, it follows from Theorem \ref{thm:BJS} (1) and Lemma \ref{lem:BJS-iota} (1) that $H_k^{\slash r-1}\neq 0$ implies $h(k,r) \in\interval{p-1}$ and $h(k,r-1)>r$.

Since $h(k,r-1)=h(k,r)+1$, the two conditions $h(k,r) \in\interval{p-1}$ and $h(k,r-1)>r$ can be combined to $h(k,r)\in\interval{r,p-1}$.  Using the definition of $h(k,r)$ and rearranging the inequalities $r\leq h(k,r)\leq p-1$ yield $n-2(k-a)\in\interval{a,p-2}$ in the claim.

Consider now the case when $h(k,r)=p-1$, i.e. $2k-n=r-1=a\in\interval{0,p-2}$, it  follows from  Lemma \ref{BJS-supplement} (and Theorem \ref{thm:BJS} (2)) that ${}^rH_k = D^{(k,k-a)} = D^{(k, n-k)}$.  Since $h(k,r-1)=h(k,r)+1$, it follows from Theorem \ref{thm:BJS} (1) that ${}^rH_{k-1}=0$, and so \[
H_k^{\slash a}=\mathrm{Coker}(\overline{\iota}_{k}^{(r-1)}) \cong {}^rH_k \cong D^{(n-(k-a),k-a)} \cong D^{(k,n-k)}.
\]

By repeatedly applying Lemma \ref{lem:del-on-slash}, we then obtain a sequence of isomorphisms
\[
H_{k}^{\slash 2k-n} \xrightarrow{\;\,\overline{\varphi}^{\slash}\;}  H_{k-1}^{\slash 2k-n-1} \xrightarrow{\;\,\overline{\varphi}^{\slash}\;} \;\; \cdots \;\;\xrightarrow{\;\,\overline{\varphi}^{\slash}\;}   H_{n-k+1}^{\slash 1} \xrightarrow{\;\,\overline{\varphi}^{\slash}\;} H_{n-k}^{\slash 0}.
\]
Now the claim for the slash homology follows by reordering the indices (that is, mapping $k$ in the above sequence to $n-k$ everywhere).

The claim for the backslash homology follows immediately from the form of the slash homology.
\end{proof}
The groups ${}^rH_k$ when $h(k,r)\geq r$ can be obtained by inductively  direct summing ${}^{r-1}H_k$ with ${}^{r+1}H_{k+1}$ and then removing the common direct summand ${}^{r}H_{k+1}$.  For $1\leq h(k,r)<r$, one use Lemma \ref{lem:BJS-iota} (2) to `flip' the groups from the other half.
Let us demonstrate this in the following example.

\begin{example}
We refer again to Figure \ref{fig:n,p=12,7} for the case $(n,p)=(12,7)$.  All the points in the plot represent the positions where ${}^{a+1}H_k$ (hence, $H_k^{\slash a}$ and $H_k^{\backslash p-2-a}$) is non-zero.  The horizontal arrows are the induced $p$-differential of the slash homology as well as the $p$-homology.  The vertical arrows are the $\iota_k^{(a+1)}$ maps on the $p$-homology groups.  The black dots, including those labelled $D^{(n-k,k)}$ are the points where the slash homology is non-zero.  The slash homologies are simply the horizontal lines consisting of these points.

As we have mentioned above, the $p$-homology groups at these points can be obtained inductively starting from $h(k,a+1)=p-1$.  For example, ${}^3H_6 = {}^2H_6\oplus {}^4H_7/{}^3H_7 = D^{(6,6)}\oplus D^{(7,5)}\oplus D^{(8,4)}$.  The $p$-homology groups at the crossed dots can be obtained by looking at the by reflecting those from $h(k,a+1)\in\interval{a+1,p-1}$ along the horizontal line that divides the black and the crossed dots.  For example, ${}^4H_5 \cong {}^1H_5 = D^{(7,5)}$.
\end{example}

\section{Tableaux combinatorics for permutation modules}\label{sec:tab}

In this section, we review some basic tableaux combinatorics that arise from permutation modules.  For simplicity, we will mostly focus only on the case when the associated partition or composition is two-row; for more general cases, see, for example, \cite{James3}.

Let us start by establishing the convention on some symbols.

For a finite set $X$, denote by $\mathfrak{S}_{X}$ the group of symmetries on $X$.  Clearly, $\mathfrak{S}_{X}$ is isomorphic to the symmetric group $\mathfrak{S}_{r}$ of rank $r:=|X|$.  For a subgroup $G\leq \mathfrak{S}_X$, and a subset $Y\subset X$, the stabiliser (subgroup) of $Y$ is given by the set of $g\in G$ such that $(Y)g=Y$ (instead of $(y)=y$ for all $y\in Y$); we also call this the $Y$-stabiliser if the context is clear.

For convenience, given a $\F$-basis $\mathcal{B}$ of a module $M$, we say that $b\in\mathcal{B}$ is a \emph{support of $v\in M$} if the coefficient of $b$ in $v$ is non-zero.

\medskip

Let $\tbx$ be a (Young) tableau.  Denote by $\tbx_{i,j}$ the content of the box located at (row, column)$=(i,j)$, with row counting from top to bottom and column counting from left to right.
From now on, we will only work with tableaux that have at most two rows.

\begin{definition}\label{def:tab}
For a $(n-k,k)$-tableau $\tbx$ is \emph{row-strict} (resp. \emph{column-strict}) if the entries in each row (resp. column) are arranged in increasing order.  $\tbx$ is \emph{standard} if it is both row- and column-strict.  We denote by $\ST_n(k)$ the set of standard Young tableaux of shape $(n-k,k)$.
\end{definition}

Recall the following relation between row-equivalent classes (a.k.a. \emph{tabloids}) of $(n-k,k)$-tableaux and $k$-subsets:
For a tableau $\tbx$, its \emph{associated $k$-subset}, denoted by 
\[\{\tbx\}:=\{\tbx_{2,1}, \tbx_{2,2}, \ldots, \tbx_{2,k}\} \in \Omega_k,\]
consist of all the elements in its second row.
Conversely, given a $k$-subset $\omega$, then we have a row-strict tableau $\tbx^\omega$ whose first row consists of elements of the complement of $\omega$ in $\interval{n}$ and second row consists of elements of $\omega$, where both rows are arranged in increasing order.
Clearly, we have $\omega = \{\tbx^\omega\}$, and so we call $\tbx^\omega$ is the \emph{tableau associated to a $k$-subset}.

\begin{example}
Take $n=8$, $k=3$, and $\omega=\{2,3,7\}$, then the (row-strict) tableau $\tbx^\omega$ associated to $\omega$ is
\begin{center}\begin{tikzpicture}[scale=0.4]\drawtbx{{1,4,5,6,8},{2,3,7}}{(0,0)}
\end{tikzpicture}.\end{center}
\end{example}

As we have mentioned in the introduction, $\F\Omega_k$ is the same as the permutation module associated to the Young subgroup $\mathfrak{S}_{n-k}\times \mathfrak{S}_k$.  Indeed, the row-equivalence class of a $(n-k,k)$-tableau, which correspond to a $k$-subset as we have just explained, is simply a shorthand of writing a coset of $\mathfrak{S}_{n-k}\times \mathfrak{S}_k$ in $\mathfrak{S}_n$.

Let us now look at an important submodule of $\F\Omega_k$.

\begin{definition}\label{def:poly}
For an $(n-k,k)$-tableau $\tbx$, denote by $C_{\tbx}\leq \mathfrak{S}_n$ the column stabiliser of $\tbx$, i.e. the direct product 
\[
\prod_{j=1}^k \langle (\tbx_{1,j}, \tbx_{2,j}) \rangle.\]
The \emph{polytabloid associated to $\tbx$} is the element 
\[
\bfe_{\tbx} := \{\tbx\}\sum_{ \sigma\in C_{\tbx}} \mathrm{sgn}(\sigma)\sigma \quad \in \F\Omega_k,\]
where $\mathrm{sgn}(\sigma)$ denotes the sign of the permutation $\sigma\in C_\tbx$.  We say that $\bfe_{\tbx}$ is \emph{standard} if so is $\tbx$.
\end{definition}

Note that a polytabloid is dependent on the defining tableau, i.e. row-equivalent but distinct tableaux can define distinct polytabloids.
\begin{example}
For $n=8$, let $\tbx$ be a tableau associated to $\interval{3}$ and $\tbx'$ be the tableau 
\begin{center}
\begin{tikzpicture}[scale=0.4]
\node at (-2,0) {$\tbx=\tbx^{\interval{3}}=$};
\drawtbx{{4,5,6,7,8},{1,2,3}}{(0,0.5)}
\node at (8,0) {and $ \tbx'=$};
\drawtbx{{8,7,6,5,4},{1,2,3}}{(9.5,0.5)}
\end{tikzpicture}.
\end{center}
The two tableaux are row-equivalent, but the polytabloid associated to $\tbx$ and $\tbx'$ are
\begin{align*}
\bfe_{\tbx} &= \{1,2,3\} -\{1,2,6\} - \{1,3,5\} + \{1,5,6\} - \{2,3,4\}  + \{2,4,6\} +\{3,4,5\} -\{4,5,6\} \\
\text{and }\; \bfe_{\tbx'} &= \{1,2,3\} -\{1,2,6\} - \{1,3,7\} + \{1,6,7\} - \{2,3,8\} + \{2,6,8\} +\{3,7,8\} -\{6,7,8\}
\end{align*}
respectively.
\end{example}

\medskip

\begin{definition}
For an integer $k$ with $n-2k\geq 0$, the \emph{Specht module} $S^{(n-k,k)}$ associated to a partition $(n-k,k)$ of $n$ is the $\F\mathfrak{S}_n$-submodule of $\F\Omega_k$ spanned by all polytabloids, i.e.
\[
S^{(n-k,k)} := \F\text{-span}\{ \bfe_{\tbx} \mid \tbx\text{ is a $(n-k,k)$-tableau}\}.
\]
\end{definition}

Note that, if $\tbx'=\tbx \sigma$ for some $\sigma\in\mathfrak{S}_n$, then one can easily see that $\bfe_{\tbx'} = \bfe_{\tbx\sigma} = \bfe_{\tbx}\sigma$.
In particular, $S^{(n-k,k)}$ is generated (over $\F\mathfrak{S}_n$) by any single polytabloid.

In the representation theory over $\mathfrak{S}_n$, the map $\varphi$ (along with a general version in the setting of any permutation module) was exploited by James to obtain a filtration of a permutation module with subquotient being Specht modules; see \cite{James3} for details.  We will only need the following results.

\begin{theorem}\label{thm:James}
Let $k$ be an integer with $2k\in\interval{0,n}$.  Then the following hold.
\begin{enumerate}
\item {\rm\cite{James1}} If $p\neq 2$, then $S^{(n-k,k)}$ has an irreducible top $D^{(n-k,k)}=S^{(n-k,k)}/\mathrm{rad}S^{(n-k,k)}$ such that $D^{(n-k,k)}\ncong D^{(n-r,r)}$ for any $r\neq k$.

\item {\rm\cite[Theorem 9.1]{James2}} $S^{(n-k,k)}\subset \Ker(\varphi_k)$.

\item {\rm \cite[Theorem 12.1]{James3}} The multiplicity of the simple module $D^{(n-k,k)}$ in the permutation module $M^{(n-k,k)}=\F\Omega_k$ is $1$.
\end{enumerate}
\end{theorem}
\begin{remark}
For a general partition $\lambda$, $S^\lambda$ may not be indecomposable (and so may not have a simple top), but this is always the case for $\lambda=(n-k,k)$ when $p\neq 2$; see \cite[Corollary 13.18]{James3}.
Also, $D^\lambda$ is conventionally defined as a possibly-zero quotient of  $S^\lambda$.  We took the liberty of simplifying this technicality by restricting to a very specific setting.
\end{remark}

It was shown by Peel \cite{Pe} that every Specht module admits a basis given by standard polytabloids.  We will need some combinatorial tools that are used in the proof.  For convenience, we translate them to our setting of $k$-subsets.  Details of the general case can be found in most standard textbook, such as \cite[Chapter 2]{Sag} and \cite[Chapter 8, 9]{James3}.

\begin{definition}\label{def:Sa-order}
Denote by $\omega\unrhd\omega'$ if $|\omega\cap \interval{i}| \leq |\omega'\cap \interval{i}|$ for all $i\in\interval{k}$; this defines a partial order on $\Omega_k$ called the \emph{dominance order}.
\end{definition}
\begin{remark}\label{rmk:orders}
The partial order $\unrhd$ is coarser than the lexicographical order (where we regard $\Omega_k$ as the set of strictly increasing $k$-tuples of elements in $\interval{n}$), but it does not matter which order one prefer to use with respect to the next result.
\end{remark}

\begin{theorem}\label{thm:Spechtbasis}
The following hold for any integer $k$ with $2k\in\interval{0,n}$.
\begin{enumerate}[(1)]
\item {\rm\cite[7.4]{James2}} If $\tbx$ is standard, then $\{\tbx\}$ is maximal in the poset $(\mathrm{Supp}(\bfe_{\tbx}), \unrhd)$ of the supports of $\bfe_{\tbx}\in S^{(n-k,k)}$ under with the dominance order.

\item {\rm \cite{Pe}} The set $\{ \bfe_{\tbx}\mid \tbx\in \ST_n(k)\}$ of standard polytabloid forms a basis of $S^{(n-k,k)}$.
\end{enumerate}
\end{theorem}


\section{Basis of the slash homology groups}\label{sec:basis}

Our next aim is to give a combinatorial description for the basis of (a non-vanishing) $H_k^{/0}$.  In view of Theorem \ref{thm:newfinal}, we will assume the following
\[
\text{\bf Assumption}: \qquad n-2k \in \interval{0,p-2}
\]
throughout this section, unless otherwise stated.

\begin{definition}\label{def:p-std}
We call a $(n-k,k)$-tableau $\tbx$ \emph{$p$-standard} if $n-2k\in\interval{0,p-2}$, $\tbx$ is standard and satisfies $\tbx_{2,j}<\tbx_{1,j+p-2}$ for all $j\in\{1, 2, \ldots, n-k-(p-2)\}$.  Denote by $\ST_n^p(k)$ the set of $p$-standard $(n-k,k)$-tableaux.
\end{definition}
\begin{remark}
As far as we know, the terminology `$p$-standard' was first used in Kleshchev's article \cite{Kl}. The definition presented here is only for the special case of two-row partitions, see \cite{Kl} for the case of tableaux in other shapes.
\end{remark}

\begin{proposition}\label{prop:dimD}
Suppose $k$ is an integer satisfying $n-2k\in\interval{0,p-2}$.
Then $\dim_{\F} D^{(n-k,k)}$ coincides with the number of $p$-standard $(n-k,k)$-tableaux.  In particular, 
\[
\dim_{\F} D^{(n-k,k)} = \sum_{j\in \Z} \binom{n}{pj+k} - \binom{n}{pj+k-1},
\]
where the binomial coefficient $\binom{a}{b}$ with $b>a$ or $b$ negative is considered zero.
\end{proposition}
\begin{remark}
The first part of this is from Kleshchev's work \cite[Corollary 2.4]{Kl}, which is obtained by translating the crystal graph combinatorics of Mathieu to $p$-standard tableau, c.f. \cite[Proposition 2.3]{Kl} and \cite[Theorem 3]{Mat}.
The dimension formula is also obtained independently in \cite{Erd} - namely, \cite[Example 5.3]{Erd} by taking $(m, s, t, d, \delta)$ therein as $(k, n-2k, 0, n-2k+1, p-(n-2k+1))$.   Erdmann's proof uses tensor products of tilting modules over $GL_2(\F)$ and the Schur-Weyl duality; see also \cite{Mat}.
\end{remark}

Let us give another proof of the dimension formula using a different combinatorics.

\begin{proof}
There is a well-known lattice path interpretation for standard $(n-k,k)$-tableaux.  Namely, consider a lattice paths in $\Z^2$ starting at $(0,0)$, ending at $(n-k,k)$, (non-strictly) bounded by the $x$-axis and the diagonal $y=x$, and with each step of the path being either $(0,1)$ or $(1,0)$.  Then there is a corresponding standard $(n-k,k)$-tableau $\tbx$ such that $i\in\interval{n}$ is in the first row if, and only if, the $i$-th step of the path is $(1,0)$. c.f. \cite[p.4-5]{Moh}.

We claim that, under this bijection, the $p$-standard tableaux correspond to the lattice paths lying inside the region $0 \geq y-x > -(p-1)$.  Indeed, it is straightforward from the definition that a standard but non-$p$-standard tableau gives rise to a path whose $i$-th $(0,1)$-step starts at $(x,y)$ with $y\leq x-(p-1)$.  Conversely, for a path with the said property of $i$-th $(0,1)$-step, then we have $y=i-1$ and so $x\geq i+p-2$.  Since $x$ counts the number of contents in the first row of  $\tbx$ that are less than $\tbx_{2,i}$, so $x\geq i+p-2$ implies that $\tbx_{2,i}>\tbx_{1,i+p-2}$ - meaning that $\tbx$ is non-$p$-standard.

As the dimension of $D^{(n-k,k)}$ for $n-2k\in\interval{0,p-2}$ is the same as the number of lattice paths bounded in $0 \geq y-x > -(p-1)$, the claim now follows by using existing formula for lattice path enumeration - for example, taking $(m,n,s,t,k)=(n-k,k,1,p-1,j)$ in \cite[Theorem 2]{Moh}.
\end{proof}

None of the results mentioned above gives any explicit description of basis of $D^{(n-k,k)}$.  It turns out that if we view $D^{(n-k,k)}$ as $H_k^{\slash 0}$, then it is possible to use the combinatorics of $\varphi$ to find an explicit basis for $D^{(n-k,k)}$ - this is our second main theorem, c.f. the standard polytabloids basis of Specht modules (Theorem \ref{thm:Spechtbasis} (2)).

\begin{theorem}\label{thm:basis}
Suppose $k$ is an integer satisfying $n-2k\in\interval{0,p-2}$.
Then a basis of the slash homology group $H_k^{/0}$ is given by \[
\big\{\, [\bfe_\tbx]\, \big|\, \tbx \text{ is a $p$-standard tableau of shape $(n-k,k)$ } \big\},
\]
where $[v]$ denotes $v+\Image(\varphi_{k-p+1}^{p-1}) \in H_k^{/0}$ for $v\in \Ker(\varphi_k)$.
In particular, $D^{(n-k,k)}$ has a basis given by $\{ \bfe_\tbx+\mathrm{rad}S^{(n-k,k)}\mid \tbx \in \ST_n^p(k)\}$.
\end{theorem}
\begin{remark}
Note that this description works in the case when $p=2$ as well, since $H_k^{/0}=0$ and there does not exist any $p$-standard tableau.  Hence, we did not need to impose the assumption that $p\neq 2$ throughout this section.
\end{remark}

We will prove Theorem \ref{thm:basis} in what follows.
The strategy is to modify the \emph{straightening rule} - the induction process that is used to show that a Specht module admits a basis given by standard polytabloids in \cite{Pe}.  The modification are designed so that we can keep track of how ``far" away the tableaux appearing in a Garnir relation is from being a $p$-standard tableau.

\subsection{Almost standard tableaux and their elementary properties}\label{subsec:basisprelim}
In this subsection, we present the definition of the partially ordered set that allows us to prove Theorem \ref{thm:basis} by induction, and also a few techniques that we will use frequently.

\begin{definition}\label{def:aST}
Let $\tbx$ be a $(n-k,k)$-tableau.
\begin{enumerate}[(1)]

\item We say that $\tbx$ \emph{almost standard} if all of the following hold.
\begin{itemize}
\item $\tbx$ is column-strict (i.e. entries increase as we go down each column);
\item $\tbx$ is second-row-strict (i.e. entries in the second row increase as we go right).
\end{itemize}
Denote by $\aST_n(k)$ the set of almost standard $(n-k,k)$-tableaux.

\item Define a partial order $\preceq$ on $\aST_n(k)$ by declaring $\tbx\preceq\tbx'$ if one of the following conditions are satisfied.
\begin{enumerate}[(i)]
\item $\tbx=\tbx'$.
\item $|\{\tbx\}\cap\interval{i}|\leq |\{\tbx'\}\cap\interval{i}|$ for all $i\in\interval{n}$ with at least one of the inequalities being strict.
\item $\{\tbx\}=\{\tbx'\}$ and for $\interval{n}\setminus\{\tbx\} = \interval{n}\setminus\{\tbx'\} = \{x_1 < x_2 < \cdots < x_{n-k}\}$, there is some $j\in\interval{n-k}$ such that
\begin{itemize}
\item $x_i$ lies in the same position in both $\tbx,\tbx'$ for all $i<j$, and
\item $x_j=\tbx_{1,a}=\tbx_{1,b}'$ with $a>b$.
\end{itemize}
\end{enumerate}
\end{enumerate}
\end{definition}

\begin{remark}\label{rmk:poset}
(1) Condition (ii) implies that the subposet $(\ST_n(k),\preceq)$ is isomorphic to the poset of standard tabloids equipped with the (opposite of) dominance order (Definition \ref{def:Sa-order}).  For us, condition (ii) measures how far a standard tableau is from being $p$-standard.

(2) Condition (iii) is used to measure how far an almost standard tableau $\mathtt{u}$ is from being the standard tableau $\tbx^{\{\mathtt{u}\}}$.  One can observe that, for any $k$-subset $\omega$, this condition coincides with the lexicographical order on $(n-k)$-tableaux with fillings from $\interval{n}\setminus \omega$ (i.e. sequences of length $n-k$ without repeating entries).
\end{remark}

\begin{example}
For $(n,k)=(5,2)$, then $(\aST_n(k),\preceq)$ is as follows. Note that the standard $(3,2)$-tableaux are drawn in bold.
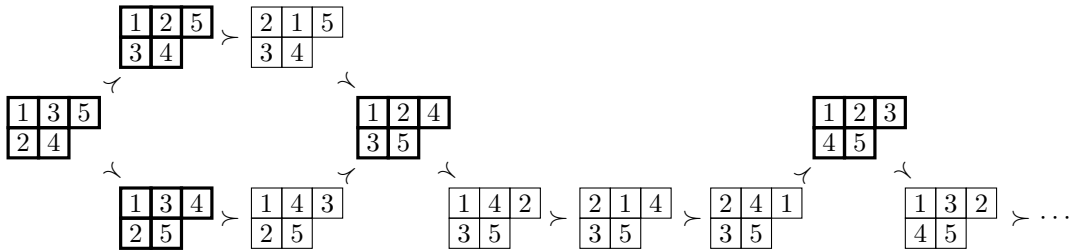
\begin{center}\begin{tikzpicture}[scale=0.4]
\node[rotate=40] at (4,1) {$\succ$};
\node[rotate=-40] at (4,-2) {$\succ$};
\node at (7.8,2.5) {$\succ$};
\node at (7.8,-3.5) {$\succ$};
\node[rotate=40] at (11.7,-2) {$\succ$};
\node[rotate=-40] at (11.7,1) {$\succ$};
\node[rotate=-40] at (15,-2) {$\succ$};
\node at (18.6, -3.5) {$\succ$};
\node at (23, -3.5) {$\succ$};
\node[rotate=45] at (26.7,-2) {$\succ$};
\node[rotate=-45] at (30,-2) {$\succ$};
\node at (33.8, -3.5) {$\succ$};
\node at (35,-3.5) {$\cdots$};

\drawtbx[very thick]{{1,3,5},{2,4}}{(0,0)}
\drawtbx[very thick]{{1,3,4},{2,5}}{(3.7,-3)} 
\drawtbx{{1,4,3},{2,5}}{(8,-3)} 
\drawtbx[very thick]{{1,2,5},{3,4}}{(3.7,3)} 
\drawtbx{{2,1,5},{3,4}}{(8,3)} 
\drawtbx[very thick]{{1,2,4},{3,5}}{(11.5,0)} 
\drawtbx[very thick]{{1,2,3},{4,5}}{(26.5,0)} 
\drawtbx{{1,4,2},{3,5}}{(14.5,-3)} 
\drawtbx{{2,1,4},{3,5}}{(18.8,-3)} 
\drawtbx{{2,4,1},{3,5}}{(23.1,-3)} 
\drawtbx{{1,3,2},{4,5}}{(29.5,-3)} 
\end{tikzpicture}\end{center}
Note that only the largest tableau $\tbx^{\{2,4\}}$ is $p$-standard when $p=3$; for any prime $p>3$, all standard tableaux are $p$-standard.
\end{example} 

The following lemma gives us greater freedom to manipulate the polytabloids; they follow immediately from the definition.

\begin{lemma}\label{lem:col-swap}
	The following hold for a $(n-k,k)$-tableau $\tbx$.
	\begin{enumerate}[(1)]
		\item Take $i\leq j\in\interval{k}$ and the following \emph{columns reordering permutation}
		\[
		\rho_{\tbx,i,j} := (\tbx_{1,j}, \tbx_{1,j-1}, \ldots, \tbx_{1,i})(\tbx_{2,j}, \tbx_{2,j-1}, \ldots, \tbx_{2,i})
		\] 
		in $\mathfrak{S}_n$, then $\bfe_{\tbx}=\bfe_{\tbx}\rho_{\tbx,i,j}$.
		
		\item If $(\tbx_{i,j})\sigma=\tbx_{i,j}$ for all $i\in\{1,2\}$ and all $j\in\interval{k}$, then $\bfe_{\tbx}=\bfe_{\tbx}\sigma=\bfe_{\tbx\sigma}$. 
	\end{enumerate}
\end{lemma}

\begin{lemma}\label{lem:bad-swap}
Consider two almost standard tableaux $\tbx,\tbx'\in \aST_n(k)$.
If $\{\tbx\}$ and $\{\tbx'\}$ differ only by one element, say, $\{\tbx\}=S\sqcup\{a\}$ and $\{\tbx'\}=S\sqcup \{b\}$ with $a>b$, then $\tbx\prec \tbx'$.
\end{lemma}
\begin{proof}
Condition (ii) in Definition \ref{def:aST} holds by assumption (c.f. \cite[3.14, 3.15]{James3}).
\end{proof}
Colloquially, Lemma \ref{lem:bad-swap} says that if we replace an entry in the second row of an almost standard tableau by a smaller value, then the new almost standard tableau (after an appropriate column reordering if needed) is larger than the original one in the $\preceq$-order.

The key in showing standard polytabloids span the Specht module is to use \emph{Garnir relation} to `straighten' a non-standard tableau (and apply induction); c.f. \cite[Section 7, 8]{James3}.  Refining this will be helpful to perform induction on almost standard tableaux.

\begin{lemma}\label{lem:Garnir}
Let $\tbx$ be an almost standard $(n-k,k)$-tableau.
If there is some $j\in \interval{k}$ such that $\tbx_{1,j}>\tbx_{1,j+1}$, then $\bfe_{\tbx}=\bfe_{\mathtt{u}}-\bfe_{\mathtt{v}}$ for some  $\mathtt{u}, \mathtt{v}\succ\tbx$ with $\{\tbx\}=\{\mathtt{u}\}$.
\end{lemma}
\begin{proof}
Since $\tbx$ is non-standard, we have the following Garnir relation
\[
\bfe_\tbx = \bfe_{\mathtt{u}} - \bfe_{\mathtt{w}} ,
\]
where $\mathtt{u}:=\tbx(\tbx_{1,j}, \tbx_{1,j+1})$ and $\mathtt{w}:=\tbx(\tbx_{1,j}, \tbx_{1,j+1}, \tbx_{2,j})$. We can see immediately that $\{\mathtt{u}\}=\{\tbx\}$ and $\mathtt{u}\succ\tbx$ as required.

Let $i$ be the smallest positive integer so that $\tbx_{1,j} < \tbx_{2,i}$.  Since column-strictness implies that $\tbx_{1,j}<\tbx_{2,j}$, second-row-strictness in turn implies that $i\leq j$. So $\bfe_{\mathtt{w}} = \bfe_{\mathtt{w}}\rho_{\mathtt{w},i,j}$ by Lemma \ref{lem:col-swap} (1).

We claim that $\mathtt{v}:=\mathtt{w}\rho_{\mathtt{w},i,j}\succ \tbx$.  Indeed, by observing that $\mathtt{w}$ is column-strict, the choice of $i,j$ ensures that $\mathtt{v}\in \aST_n(k)$.  Since $\tbx_{2,j}>\tbx_{1,j}$ (by column-strictness of $\tbx$), $\{\mathtt{v}\} = \{\mathtt{w}\}$ is obtained from replacing the element $\tbx_{2,j}$ in $\{\tbx\}$ by the smaller $\tbx_{1,j}$.
Thus it follows from Lemma \ref{lem:bad-swap} that $\mathtt{v}\succ \tbx$, as required.
\end{proof}

\begin{example} Consider the almost standard tableau
\begin{center}\begin{tikzpicture}[scale=0.4]
	\node at (-0.5,-.7) {$\tbx =$};
	\drawtbx{{1,3,2},{4,5}}{(0,0)}
\end{tikzpicture}.\end{center}
Applying Lemma \ref{lem:Garnir} (with $j=2$) yields $\bfe_\tbx = \bfe_{\mathtt{u}}-\bfe_{\mathtt{w}} = \bfe_{\mathtt{u}}-\bfe_{\mathtt{v}}$, 
where the tableaux involved are
\begin{center}\begin{tikzpicture}[scale=0.4] 
	\node at (-5.2,-.5) {$\mathtt{u}=\tbx(2,3)=$};
	\drawtbx{{1,2,3},{4,5}}{(-3,0)}
	\node at (4.5,-.5) {$,\quad  \mathtt{w}=\tbx(3,2,5)=$};
	\drawtbx{{1,2,5},{4,3}}{(8,0)}
	\node at (17,-.5) {$\notin \aST_n(k), \quad \mathtt{v} =\mathtt{w}\rho_{\mathtt{w},1,2}=$};
	\drawtbx{{2,1,5},{3,4}}{(22,0)}
\end{tikzpicture}.\end{center}
It is easy to check that $\mathtt{u},\mathtt{v}\succ \tbx$.
\end{example}

Let us refine the Lemma \ref{lem:Garnir} in the following `induction-friendly' form.

\begin{lemma}\label{lem:nonstd}
If an almost standard tableau $\tbx\in \aST_n(k)$ is non-standard, then we have
\[
\bfe_{\tbx} \in \bfe_{\mathtt{s}}+\Fspan\{\bfe_{\tbx'}\mid \tbx\prec \tbx'\in \aST_n(k)\}\]
for the (unique) standard tableau $\mathtt{s}$ with $\{\mathtt{s}\}=\{\tbx\}$.
\end{lemma}
\begin{proof}
By assumption, there is some $j\in\interval{n-k}$ such that $\tbx_{1,j}>\tbx_{1,j+1}$.  Consider the case $j>k$.  Then it follows from Lemma \ref{lem:col-swap} (2) that $\bfe_{\tbx}=\bfe_{\tbx}(\tbx_{1,j},\tbx_{1,j+1})$.  By Definition \ref{def:aST}, we have $\tbx\prec \tbx(\tbx_{1,j},\tbx_{1,j+1})\in \aST_n(k)$ and the claim follows.  

Suppose now that $j\leq k$. Consider now the subset
\[
\aST_{\tbx} := \Bigl\{ \mathtt{w} \in \aST_n(k) \Bigm| \{\mathtt{w}\}=\{\tbx\} \Bigr\} \subset \aST_n(k).
\]
Recall from Remark \ref{rmk:poset} (2) $\aST_{\tbx}$ is a totally ordered set under $\preceq$, so we can show (2) by induction on $\aST_{\tbx}$. Indeed, the maximum of $\aST_{\tbx}$ is the unique standard tableau $\mathtt{s}$ associated to $\{\tbx\}$ and there is nothing to show. Otherwise, it follows from Lemma \ref{lem:Garnir} that we can write $\bfe_{\tbx} = \bfe_{\mathtt{u}} - \bfe_{\mathtt{v}}$ with $\mathtt{u}, \mathtt{v}\succ \tbx$ and $\mathtt{u}\in\aST_{\tbx}$.  Now the claim follows by applying induction hypothesis on $\mathtt{u}$.
\end{proof}

\subsection{Straightening a $p$-bad tableau}\label{subsec:badtab}
For convenience, let us call such $\tbx$ a \emph{$p$-bad} (standard) tableau, and denote by
\[
\pbad_n(k) := \ST_n(k)\setminus \ST_n^p(k)
\]
the set of all $p$-bad tableaux.
The aim of this subsection is to setup the inductive argument (Lemma \ref{lem:badstd}) needed to show that the $\Image(\varphi^{p-1})$-cosets representative given by the $p$-standard polytabloids spans $H_k^{\slash 0}$.
Along the way, we will derive a $p$-standard version of Garnir relation that straighten a (polytabloid associated to a) $p$-bad tableau, in analogous to the straightening of non-standard polytabloid.

Let us start by introducing some combinatorial gadgets to aid exposition.  First note that there is a bilinear operation on  $\bigoplus_{k=0}^n\F\Omega_k$ given by bilinearly extending 
\[
	\nu\cdot \omega := \begin{cases} \nu\cup \omega, &\text{if } \nu\cap \omega=\emptyset; \\
	0, &\text{otherwise}
\end{cases}
\]
for $\nu\in\Omega_k$ and $\omega \in\Omega_\ell$.  

Another tool we need is the divided powers of $\varphi$, i.e. $\varphi^{(a)}:=(a!)^{-1}\varphi^a$ for $a\in\interval{0,p-1}$.  In other words, $\varphi_k^{(a)}$ sends a $k$-subset $\omega$ to all formal sum of $(k-a)$-subsets contained in $\omega$.  This map is called \emph{$a$-step boundary map} in \cite{Wil} and we will follow this convention.  The multi-step boundary maps satisfy the following (super-)Leibniz rule, which is called splitting lemma in \cite{Wil}.

\begin{lemma}{\rm \cite[Lemma 3.5]{Wil}, \cite[Lemma 2.1]{BJS}}\label{lem:splitting}
For $v\in \F\Omega_k$ and $w\in \F\Omega_\ell$ such that $\nu\cap\omega=\emptyset$ for any support $\nu$ of $v$ and any support $\omega$ of $w$, we have
\[
(v\cdot w)\varphi_{k+\ell}^{(t)} = \sum_{i=0}^t((v)\varphi_k^{(i)})\cdot((w)\varphi_\ell^{(t-i)}).
\]
\end{lemma}

\begin{remark}
Note that we could have defined $\varphi^{(a)}$ for all $a\geq 0$ if we replace the ground field $\F$ by $\Z$.  We suspect that $\varphi$ comes from (or defines) a `gaea structure', in the sense of \cite{EQ}, on the Schur algebra $S_{\Z}(2,n)$, i.e. the endomorphism algebra of the $\Z\mathfrak{S}_n$-module $\bigoplus_{k=0}^n \Z\Omega_k$.
\end{remark}

The following key lemma provides a relation between certain polytabloids and the map $\varphi$.

\begin{lemma}\label{lem:tbx-reln}
Suppose $\tbx$ is a $(n-k,k)$-tableau.
For an integer $i\in \interval{(n-k)-(p-2)}$, consider the subsets
\[
B:=\{\tbx_{1,i+j}\mid j\in\interval{0,p-2} \} \subset K:=B\cup\{\tbx_{2,i+j}\mid j\in\interval{0,\min\{p-2, k-i\}}\}.
\] 
Then the following equation holds.
\[
\bfe_\tbx \sum_{\tau\in \mathfrak{S}_B} \tau = (\bfe_{\tbx'}\cdot K)\varphi^{p-1},
\]
where $\tbx'$ is the tableau obtained from $\tbx$ by removing the boxes with entries in $K$ and attach them to the end of first row (see Figure \ref{fig:tbx}).
\end{lemma}
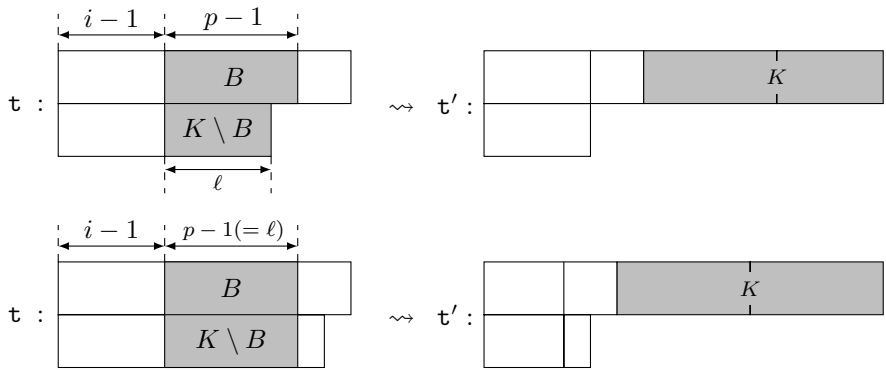
\begin{figure}[!hbtp]
\centering
\begin{tikzpicture}[scale=0.7]
\node at (-0.6,1) {$\mathtt{t}$\;\;:};
\draw  (0,1) rectangle (2,0)  (0,2) rectangle (5.5,1);
\filldraw [fill=lightgray]  (2,2) rectangle (4.5,1);\node at (3.3,1.5) {$B$};
\filldraw [fill=lightgray]  (2,1) rectangle (4,0);\node at (3,0.5) {$K\setminus B$};
\draw[latex-latex] (2,2.3) -- (4.5,2.3);\draw[dashed] (2,2) -- +(0,0.7) (4.5,2)--+(0,0.7);\draw (3.3,2.6) node {$p-1$};
\draw[latex-latex] (2,2.3) -- (0,2.3);\draw[dashed] (0,2) -- +(0,0.7);\draw (1,2.6) node {$i-1$};
\draw[latex-latex] (2,-0.25) -- (4,-0.25);\draw[dashed] (2,0) -- +(0,-0.7) (4,0)--+(0,-0.7);\draw (3,-0.5) node {\footnotesize $\ell$};

\node at (7,1) {$\leadsto\;\; \mathtt{t'}$\;:};
\draw (8,1) rectangle (10,0) (8,2) rectangle (11,1);
\filldraw [fill=lightgray]  (11,2) rectangle (13.5,1) (13.5,2) rectangle (15.5,1);
\node [fill=lightgray] at (13.5,1.5) {\footnotesize $K$};
\draw (10,1) -- (10,2);

\begin{scope}[shift={(0,-4)}]
\node at (-0.6,1) {$\mathtt{t}$\;\;:};
\draw  (0,1) rectangle (5,0) (0,2) rectangle (5.5,1);
\filldraw [fill=lightgray]  (2,2) rectangle (4.5,0);\node at (3.25,1.5) {$B$};\node at (3.25,0.5) {$K\setminus B$};
\draw (2,1) -- +(3,0);
\draw[latex-latex] (2,2.3) -- +(2.5,0);\draw[dashed] (2,2) -- +(0,0.7) (4.5,2)--+(0,0.7);\draw (3.3,2.6) node {\footnotesize $p-1(=\ell)$};
\draw[latex-latex] (2,2.3) -- (0,2.3);\draw[dashed] (0,2) -- +(0,0.7);\draw (1,2.6) node {$i-1$};

\node at (7,1) {$\leadsto\;\;\mathtt{t'}$\;:};
\draw (8,1) rectangle (9.5,0) (9.5,1) rectangle (10,0) (8,2) rectangle (10.5,1);
\filldraw [fill=lightgray]  (10.5,2) rectangle (13,1) (13,2) rectangle (15.5,1);
\node [fill=lightgray] at (13,1.5) {\footnotesize $K$};
\draw (9.5,1) -- (9.5,2);
\end{scope}
\end{tikzpicture}
\caption{Notations used in Lemma \ref{lem:tbx-reln}.}\label{fig:tbx}
\end{figure}
\begin{proof}

Let us first set up some additional notations.  Take $\ell := |K \cap \{\tbx\}|=\min\{p-2, k-i\}+1\in\interval{p-1}$.  Hence, $\tbx'$ is a $(n-(k-\ell), (k-\ell))$-tableau, c.f. Figure \ref{fig:tbx}.

We will prove the lemma by showing the following two equations:
\begin{enumerate}[(i)]
\item $\bfe_\tbx \sum_{\tau\in \mathfrak{S}_B} \tau = -\bfe_{\tbx'}\cdot(K)\varphi^{(p-1)}$.
\item $-\bfe_{\tbx'}\cdot(K)\varphi^{(p-1)} = (\bfe_{\tbx'}\cdot K)\varphi^{p-1}$.
\end{enumerate}

\underline{Proof of (i)}: 
For clarity, what we need to show is the equality
\begin{equation}\label{eq:countB}
\sum_{\tau\in \mathfrak{S}_B} \sum_{\sigma\in C_{\tbx}} \sgn(\sigma)\{\tbx\}\sigma\tau = -\sum_{\sigma' \in C_{\tbx'}} \sgn(\sigma')\big(\{\tbx'\}\sigma'\big) \cdot (K)\varphi_{\ell+p-1}^{(p-1)}.
\end{equation}

Consider a $k$-subset $\omega$ that is of the form $\{\tbx\}\sigma\tau$ with $\sigma\in C_\tbx$ and $\tau\in \mathfrak{S}_B$. Since $\{\tbx'\}\subset \{\tbx\}$, $\sigma$ induce a unique element $\sigma'\in C_{\tbx'}$ that fixes $\tbx'_{1,j}$ for all $j>k-\ell$ and
\[\omega = \{\tbx\}\sigma\tau = (\{\tbx'\}\sigma')\cdot (K\cap \omega).
\]

Moreover, we can observe that (c.f. Figure \ref{fig:tbx}) the $\ell$-subset $K\cap\omega$ must appear as a support of $(K)\varphi_{\ell+p-1}^{(p-1)}$.
Hence, $\omega$ must arise as a support in the right-hand side of \eqref{eq:countB}.
Note that, conversely, every support in the right-hand side can always be obtained in such a form.

Now we need to show that the coefficients $\lambda_\omega$ and $\mu_\omega$ of $\omega$ in the left-hand side and right-hand side of \eqref{eq:countB}, respectively, coincide.

The number of times that $\omega$ is picked up in the summation over $\mathfrak{S}_B$ is the order of the stabiliser subgroup of $\omega$ in $B$.
Since this group is just $\mathfrak{S}_{(\interval{n}\setminus\omega)\cap B} \times \mathfrak{S}_{\omega\cap B}$, the required number is $(p-1-r)!r!$, where $r=|\omega\cap B|$.
It follows from Wilson's theorem that $(p-1-r)!r!\equiv (-1)^{r+1}$ (mod $p$), so we have $\lambda_\omega = (-1)^{r+1}\sgn(\sigma)$.
Since $\sigma(\sigma')^{-1}$ acts only on columns in $\tbx$ with contents in $K$, we have
\begin{align*}
\lambda_\omega &= (-1)^{r+1}\sgn(\sigma)\\
&= (-1)^{r+1}\sgn(\sigma')\sgn(\sigma(\sigma')^{-1}) \\
&=  (-1)^{r+1}\sgn(\sigma')(-1)^{|\omega\cap B|} \;\;=\;\; (-1)^{2r+1}\sgn(\sigma') \;\; = \;\; -\sgn(\sigma')
\end{align*}

For the coefficient $\mu_\omega$, since the coefficient of every support of $(K)\varphi_{\ell+p-1}^{(p-1)}$ is $1$, combining with the fact that $\omega = (\{\tbx'\}\sigma')\cdot (K\cap\omega)$ yields $\mu_\omega = -\sgn(\sigma')$.
This finishes the proof of (i).

\underline{Proof of (ii)}:
It is clear from the construction that every $k$-subset in the support of $\bfe_{\tbx'}$ is disjoint from $K$, so it follows from Lemma \ref{lem:splitting} that 
\[
(\bfe_{\tbx'}\cdot K)\varphi^{(p-1)} = \sum_{i=0}^{p-1} (\bfe_{\tbx'})\varphi^{(i)}\cdot (K)\varphi^{(p-1-i)}.
\]
Since $\tbx'$ is a $(n-(k-\ell), k-\ell)$-tableau, it follows from Theorem \ref{thm:James} (2) that $\varphi^i$ annihilates $\bfe_{\tbx'}$ for any $i>0$.  This finishes the proof of (ii).

The required equality now follows from the observation that $\varphi^{(p-1)}=(p-1)!\varphi^{p-1}=-\varphi^{p-1}$, where the last equality follows from Wilson's theorem.
\end{proof}

Before proceeding to the next step, we need an elementary observation.

\begin{lemma}\label{lem:badpaths}
For a standard tableau $\tbx\in\ST_n(k)$, the following are equivalent
\begin{enumerate}[(1)]
\item $\tbx$ is $p$-bad.
\item There is a strictly descending chain $\tbx_{2,i} > \tbx_{1,i+p-2} > \tbx_{1,i+p-3} > \cdots > \tbx_{1,i}$ for some $i\in\interval{k}$.
\item $\tbx_{2,i}\geq p+2(i-1)$ for some $i\in \interval{k}$.
\end{enumerate}
\end{lemma}
\begin{proof}
\underline{(1) $\Leftrightarrow$ (2)}:  By definition.

\underline{(2) $\Rightarrow$ (3)}:  Since $\tbx$ is standard, the smallest possibly entry for $\tbx_{1,i+p-2}$ is $i+p-2$.  On the other hand, $\tbx_{2,i}$ must be larger than $\tbx_{2,j}$ for all $j<i$.  Hence $\tbx_{2,i} > (i+p-2)+(i-1) = 2(i-1)+p-1$.

\underline{(3) $\Rightarrow$ (1)}:  Strictness of $\tbx$ implies that $\tbx_{2,i}\geq 2i$, and so one needs to at least fill up $(2i-1)+(p-1)-2i-1 = p-2$ more boxes after the $i$-th box in the first row before filling $i$-th box in the second row by $\tbx_{2,i}$.  Hence, we have $\tbx_{2,i}>\tbx_{1,i+p-2}$ as required.

\end{proof}

Lemma \ref{lem:badpaths} allows us to introduce the following terminology.
\begin{definition}\label{def:Bi}
Let $\tbx\in\pbad_n(k)$ be a $p$-bad standard $(n-k,k)$-tableau.
We say that $\tbx_{2,i}$ is a \emph{bad entry} if $\tbx_{2,i}\geq p+2(i-1)$.
Since such an entry always exists by Lemma \ref{lem:badpaths}, we can  define the following $(p-1)$-subset of $\interval{n}$:
\[
B_i:=\{\tbx_{1,i+1} < \tbx_{1,i+2} < \cdots < \tbx_{1,i+p-1} < \tbx_{2,i}\}.
\]
\end{definition}

\begin{lemma}\label{lem:tbx-out}
Let $\tbx\in \pbad_n(k)$ be a $p$-bad standard tableau with a bad entry $\tbx_{2,i}$.
Then the element $\bfe_{\tbx}\sum_{\tau\in \mathfrak{S}_{B_i}} \tau \in S^{(n-k,k)}\subset \F\Omega_k$ belongs to the submodule $\Image(\varphi_{k+p-1}^{p-1})$.
\end{lemma}
\begin{proof}
Take $B:=B_i(\tbx_{1,i},\tbx_{2,i})$, then we have $(\tbx_{1,i},\tbx_{2,i})\mathfrak{S}_{B_i}(\tbx_{1,i},\tbx_{2,i}) = \mathfrak{S}_B$.
This implies that
\begin{align*}
\bfe_{\tbx}\sum_{\tau\in\mathfrak{S}_{B_i}}\tau  &=  \Big(\bfe_{\tbx}(\tbx_{1,i},\tbx_{2,i})\Big) \sum_{\sigma \in \mathfrak{S}_B} \sigma (\tbx_{1,i},\tbx_{2,i}) = -\left(\bfe_{\tbx}\sum_{\sigma \in \mathfrak{S}_B} \sigma\right) (\tbx_{1,i},\tbx_{2,i}),
\end{align*}
where the last equality follows from the observation that $\bfe_\tbx = -\bfe_\tbx(\tbx_{1,i},\tbx_{2,i})$.
On the other hand, it follows from Lemma \ref{lem:tbx-reln} that $\bfe_{\tbx}\sum_{\sigma\in\mathfrak{S}_B}\sigma \in \Image(\varphi^{p-1})$.
Hence, the element on the right-hand side is also in $\Image(\varphi^{p-1})$, and so is $\bfe_\tbx\sum_{\tau\in \mathfrak{S}_{B_i}} \tau$.
\end{proof}

This Lemma \ref{lem:tbx-out} can be viewed as the Garnir relation in the context of straightening a $p$-bad standard tableau to a $p$-standard (c.f. the first line of the proof of Lemma \ref{lem:badstd} versus the first line of the proof of Lemma \ref{lem:Garnir}).

\begin{example}
Consider the standard tableau
\begin{center}\begin{tikzpicture}[scale=0.4]
		\node at (-1,-.7) {$\tbx =$};
		\drawtbx{{1,2,5},{3,4}}{(0,0)}
\end{tikzpicture}.\end{center}
Observe that $\tbx$ is $3$-bad with bad entry $\tbx_{2,1}$. Following the notation of Lemma \ref{lem:tbx-out}, we have $B_i = B_1 = \{ 2,3\}$, and $B=B_1(1,3)=\{1,2\}$.  Calculating $\bfe_\tbx\sum_{\tau\in \mathfrak{S}_{\{2,3\}}}\tau$ explicitly yields
\begin{align*}
\bfe_\tbx\sum_{\tau\in \mathfrak{S}_{\{2,3\}}}\tau &= \bfe_\tbx + \bfe_\tbx(2,3) \\
&= \Big(\{3,4\}-\{1,4\}-\{2,3\}+\{1,2\}\Big) + \Big(\{2,4\}-\{1,4\}-\{2,3\}+\{1,3\}\Big) \\
&= \{1,2\}+\{1,3\}+\{2,4\}+\{3,4\} - 2\Big( \{1,4\}+\{2,3\}\Big)\\
&= (\{1,2,3,4\})\varphi^{(2)}\\
&= (2!)^{-1}(\{1,2,3,4\})\varphi^2 = -(\{1,2,3,4\})\varphi^2.
\end{align*}
Note that the third and last equality follows from $2!=2=-1$ in a field of characteristic 3, c.f. Proof of Claim (i) in the proof of Lemma \ref{lem:tbx-reln}.

This is a small example where we can spot the fourth equality quite easily.  The `correct' way to see why $\{1,2,3,4\}$ appears is to follow the proof, which says that it arises from the formula $\bfe_\tbx \sum_{\sigma\in \mathfrak{S}_{B}}\sigma = (\bfe_{\tbx'}\cdot K)\varphi^{p-1}$ of Lemma \ref{lem:tbx-reln}.  Since we have $B=B_1(1,3)=\{1,2\}\subset K=\{1,2,3,4\}$ and $\tbx'$ being a tableau of shape $(5,0)$, we have $\bfe_{\tbx'}=\{\tbx'\}=\emptyset$ and $\bfe_{\tbx'}\cdot K = \{1,2,3,4\}$, as required.
\end{example}

\begin{lemma}\label{lem:badpaths2}
Let $\tbx\in \pbad_n(k)$ be a $p$-bad standard tableau with a bad entry $\tbx_{2,i}$.
Then $\tbx\sigma$ is column-strict for any $\sigma\in B_i$.
In particular, if $(\tbx_{2,i})\sigma=\tbx_{2,i}$, then $\tbx\sigma$ is almost standard.
\end{lemma}
\begin{proof} 
Column-strictness follows from the existence of the chain in Lemma \ref{lem:badpaths}.
More precisely, for $j \in \interval{p-2}$ with $i+j\leq k$, we have $(\tbx\sigma)_{2,i+j}=\tbx_{2,i+j}>\tbx_{2,i}\geq \tbx_{1,j+l}$ for all $l\in\interval{p-2}$ and so $\tbx_{2,i+j}>(\tbx\sigma)_{1,i+j}$ always.
For the $i$-th column, again the chain in Lemma \ref{lem:badpaths} tells us that $(\tbx_{2,i})\sigma\geq\tbx_{1,i+1}>\tbx_{1,i}=(\tbx\sigma)_{1,i}$.
The last statement follows from the fact that $\tbx$ is standard and the assumption that $\sigma$ fixes all entries in the second row.
\end{proof}

Recall that
\[\ [v]:=v+\Image(\varphi^{p-1}) \in H_k^{/0}\]
denotes the coset containing $v\in \Ker(\varphi_k)$ in the quotient $H_k^{/0}$.

\begin{lemma}\label{lem:badstd}
$[\bfe_{\tbx}] \in \Fspan\{[\bfe_{\tbx'}]\mid \tbx\prec\tbx'\in\aST_n(k)\}\subset H_k^{/0}$ for any $\tbx\in \pbad_n(k)$.
\end{lemma}
\begin{proof}
By Lemma \ref{lem:tbx-out}, if $\tbx_{2,i}$ is a bad entry (which always exists by Lemma \ref{lem:badpaths}), we have 
\[
-[\bfe_\tbx]=\sum_{1\neq \sigma\in \mathfrak{S}_{B_i}} [\bfe_{\tbx}\sigma]
\]
in $H_k^{/0}$.
So it suffices to show that $[\bfe_\tbx\sigma]$ with non-identity $\sigma\in \mathfrak{S}_{B_i}$ is in $\Fspan\{\bfe_{\tbx'}\mid \tbx\prec\tbx'\in\aST_n(k)\}$.

We partition $\mathfrak{S}_{B_i}\setminus\{1\}$ into two disjoint subsets $S\sqcup T$ so that $S$ consists of all the non-identity permutations that fix $\tbx_{2,i}$.
By Lemma \ref{lem:nonstd} (2), for any $\sigma\in S$, we have
\[
\bfe_{\tbx}\sigma \in \bfe_{\tbx} + \Fspan\{ \tbx' \in \aST_n(k) \mid \tbx\prec\tbx'\in\aST_n(k) \}.
\]
Therefore, we have $\sum_{\sigma\in S} \bfe_\tbx \sigma = |S|\bfe_\tbx + v$ for some $v\in \Fspan\{\bfe_{\tbx'}\mid \tbx'\succ\tbx\}$.
Note that the subgroup $\{1\}\sqcup S$ of $\mathfrak{S}_{B_i}$ that stabilises $\tbx_{2,i}$ is of order $(p-2)!$, which is congruent to $1$ mod $p$ by Wilson's theorem.
This means that $|S|$ is congruent to $0$ mod $p$, which yields
\[
-[\bfe_\tbx] = (|S|[\bfe_\tbx]+[v])+\sum_{\sigma \in T}[\bfe_{\tbx}\sigma] =  [v]+\sum_{\sigma \in T}[\bfe_{\tbx}\sigma].
\]

Now consider $\sigma\in T\subset \mathfrak{S}_{B_i}$.  
It follows from Lemma \ref{lem:badpaths2} that $\tbx\sigma$ is column-strict.
By Lemma \ref{lem:col-swap}, we can apply a column reordering permutation $\rho$ so that $\tbx\sigma\rho$ is almost standard and $\bfe_\tbx\sigma=\bfe_\tbx\sigma\rho$.
Since
\[
\{\tbx \sigma\rho\}=\{\tbx\sigma\} = \big(\{\tbx\}\setminus\{\tbx_{2,i}\}\big)\sqcup\{(\tbx_{2,i})\sigma\},
\]
and Lemma \ref{lem:badpaths} says that $(\tbx_{2,i})\sigma< \tbx_{2,i}$, applying Lemma \ref{lem:bad-swap} yields $\tbx\sigma\rho\succ \tbx$.

Thus, we have $\sum_{\sigma\in T}[\bfe_\tbx\sigma]\in \Fspan\{\tbx'\mid \tbx\prec\tbx'\in\aST_n(k)\}$, and the proof is now complete.
\end{proof}

\medskip

All the required ingredients for proving Theorem \ref{thm:basis} are now ready.

\subsection{Proof of Theorem \ref{thm:basis}}\label{subsec:basisproof}

Recall from Theorem \ref{thm:newfinal} that $H_k^{\slash 0}\cong D^{(n-k,k)}$.  Since this is the unique composition factor of $\F\Omega_k$ appearing as the top of the submodule $S^{(n-k,k)}$ by Theorem \ref{thm:James} (1) and (3), we only need to find the minimal spanning set of $H_k^{\slash 0}$ contained in $\mathcal{B}:=\{[\bfe_\tbx]\mid \tbx\in\ST_n(k)\}$.

Before proceeding, we remark that $\Ker(\varphi_k)$ is in general larger than $S^{(n-k,k)}$.  In fact, \cite[Theorem 5.3]{BJS} shows that $\Ker(\varphi_k)$ is generated (as $\F\mathfrak{S}_n$-module) by $\Image(\varphi^{p-1})$ along with any (single) polytabloid.

To prove that any $[\bfe_\tbx]\in \mathcal{B}$ is in the $\Fspan$ of $\{[\bfe_{\tbx'}] \mid \tbx'\in \ST_n^p(k)\}$, it is enough to replace $\mathcal{B}$ by the following larger set.

\begin{lemma}
For any almost standard tableau $\tbx\in \aST_n(k)$, we have
\[
\ [\bfe_{\tbx}]:=\bfe_{\tbx}+\Image(\varphi^{p-1}) \;\;\in \;\; \Fspan\{[\bfe_{\tbx'}]\mid \tbx' \text{ is $p$-standard }\}.
\]
\end{lemma}
\begin{proof}
We will show this by induction on $(\aST_n(k),\preceq)$.
Note that a maximal (in fact, the maximum) element of this poset is the standard tableau $\tbx^\omega$ associated to $\omega=\{2, 4, \ldots,2k\}$, which is always $p$-standard unless $p = 2$.  Clearly, there is nothing to show in the case when $\tbx=\tbx^\omega$.

Consider now $\tbx\prec\tbx^\omega$.
If $\tbx$ is a $p$-standard tableau, then there is nothing to show; otherwise, $\tbx$ is either $p$-bad or non-standard.

If $\tbx$ is $p$-bad, then it follows from Lemma \ref{lem:badstd} that $[\bfe_\tbx]\in \Fspan\{[\bfe_{\tbx'}] \mid \tbx \prec \tbx'\in\aST_n(k)\}$.

Consider the case when $\tbx$ is non-standard.  Since by definition the unique standard tableau $\mathtt{s}$ with $\{\mathtt{s}\}=\{\tbx\}$ satisfies $\mathtt{s}\succ \tbx$, Lemma \ref{lem:nonstd} implies that $[\bfe_\tbx]\in \Fspan\{[\bfe_{\tbx'}]\mid \tbx\prec\tbx'\in\aST_n(k)\}$.
By the induction hypothesis, every $[\bfe_{\tbx'}]$ with $\tbx'\succ\tbx$ are in $\Fspan\{ [\bfe_{\mathtt{u}}] \mid \mathtt{u}\in \ST_n^p(k)\}$, so it follows immediately that so is $[\bfe_\tbx]$.
\end{proof}

Recall from Proposition \ref{prop:dimD} that $D^{(n-k,k)}\cong H_k^{\slash 0}$ has dimension given by the number of $p$-standard tableaux, so $\Fspan\{[\bfe_\tbx]\mid \tbx\in\ST_n^p(k)\}$ is the minimal spanning set for $H_k^{\slash 0}$, and hence a basis.

Finally, we note that $D^{(n-k,k)}$ is the irreducible top of $S^{(n-,k)}$ (Theorem \ref{thm:James} (1)) and so the description of basis for $H_k^{\slash 0}$ translates to that of $D^{(n-k,k)}$ as claimed.

The proof of Theorem \ref{thm:basis} is now complete.


\begin{thebibliography}{99999}
\bibitem[BJS]{BJS} S.~Bell; P.~Jones; J.~Siemons.
On modular homology in the {B}oolean algebra. {II}.
{\it J. Algebra} {\bf 199} (2), 1998, 556--580.

\bibitem[EQ]{EQ} B.~Elias; Y.~Qi.  Categorifying Hecke algebras at prime roots of unity, part I.  preprint, arXiv:2005.03128.

\bibitem[Erd]{Erd} K.~Erdmann.
Tensor products and dimensions of simple modules for symmetric groups. {\it Manuscripta Math.} {\bf 88}, 1995, 357--386.

\bibitem[J1]{James1} G.~D.~James.
The irreducible representations of the symmetric groups.  {\it Bull. London Math. Soc.} {\bf 8} (3), 1976, 229--232.

\bibitem[J2]{James2} G.~D.~James.
A characteristic-free approach to the representation theory of $\mathfrak{S}_n$.  {\it J. Algebra} {\bf 46} (2), 1977, 430--450.

\bibitem[J3]{James3} G.~D.~James.  The representation theory of the symmetric groups, Lecture Notes in Mathematics, vol. 682, Springer, Berlin, 1978.

\bibitem[Kap]{Kapranov} M.~M.~Kapranov. On the q-analog of homological algebra. preprint, arXiv:q-alg/9611005

\bibitem[KQ]{KhQi} M.~Khovanov; Y.~Qi.  An approach to categorification of some small quantum groups.  {\it Quantum Topol.} {\bf 6} (2), 2015, 185--311.

\bibitem[Kl]{Kl} A.~Kleshchev.  Completely splittable representations of symmetric groups.  {\it J. Algebra} {\bf 181} (2), 1996, 584--592.

\bibitem[Mat]{Mat} O.~Mathieu.
On the dimension of some modular irreducible representations of the symmetric group.  {\it Lett. Math. Phys} {\bf 38}, 1996, 23--32.

\bibitem[M1]{Mayer1} W.~Mayer. A new homology theory I. {\it Ann. Math.} {\bf 43}, 1942, 370–380.

\bibitem[M2]{Mayer2} W.~Mayer. A new homology theory II. {\it Ann. Math.} {\bf 43}, 1942, 594–605.

\bibitem[Mir]{Mir} D.~Mirmohades.  Homologically optimal categories of sequences lead to $N$-complexes. 2014, arXiv:1405.3921 

\bibitem[Moh]{Moh} S.~G.~Mohanty. {\it Lattice path counting and applications.} Academic Press, New York-London-Toronto, Ont., 1979. xi+185 pp. ISBN: 0-12-504050-4

\bibitem[Pe]{Pe} M.~H.~Peel. Specht modules and the symmetric groups. {\it J. Algebra} {\bf 36}, 1975, 88-97.

\bibitem[Qi]{Qi} Y.~Qi.  Hopfological algebra.  {\it Compositio Mathematica} {\bf 150} (01): 1–45, 2014.

\bibitem[Sag]{Sag} B.~E.~Sagan. {\it The Symmetric Group: Representations, Combinatorial Algorithms, and Symmetric Functions.} Graduate Texts in Mathematics 203, Springer-Verlag, New York, 2001. xvi+240 pp. ISBN: 978-1-4419-2869-6

\bibitem[Wil]{Wil} M.~Wildon. The multistep homology of the simplex and representations of symmetric groups. 2018, arXiv:1803.00465
\end{thebibliography}
\end{document}